\documentclass[reqno,11pt]{amsart}
\usepackage{amsmath, amsfonts,amsthm,amssymb,amsbsy,upref,color,graphicx,amscd,hyperref,enumerate}
\usepackage[active]{srcltx}
\usepackage[latin1]{inputenc}
\usepackage{bbm}
\usepackage{algorithm}
\usepackage{algorithmicx}
\usepackage{algpseudocode}
\usepackage{graphicx}
\usepackage{subfig} 
\usepackage{pgf}
\usepackage{xfrac}

\def\ds{\displaystyle}
\def\eps{{\varepsilon}}
\def\N{\mathbb{N}}

\def\R{\mathbb{R}}

\def\HH{\mathcal{H}}

\newcommand{\be}{\begin{equation}}
\newcommand{\ee}{\end{equation}}

\newcommand{\de}{\partial}
\newcommand{\dist}{{\rm {dist}}}
\newcommand{\spt}{{\rm {supp}}}
\newcommand{\bC}{{\bf C}}

\theoremstyle{plain}
\newtheorem{theo}{Theorem}
\newtheorem{exam}{Example}

\numberwithin{equation}{section}
\theoremstyle{plain}
\newtheorem{teo}{Theorem}[section]
\newtheorem{lemma}[teo]{Lemma}

\newtheorem{prop}[teo]{Proposition}

\theoremstyle{remark}
\newtheorem{oss}[teo]{Remark}

\title[An epiperimetric inequality for some free boundary problems]{An epiperimetric inequality for the regularity of some free boundary problems: the $2$-dimensional case}

\author{Luca Spolaor, Bozhidar Velichkov}

\address {Luca Spolaor: \newline \indent
	Massachusetts Institute of Technology (MIT), 
	\newline \indent
	77 Massachusetts Avenue, Cambridge 
	MA 02139, USA}
\email{lspolaor@mit.edu}

\address {Bozhidar Velichkov: \newline \indent
Laboratoire Jean Kuntzmann (LJK), Universit\'e Grenoble Alpes
\newline \indent
B\^atiment IMAG, 700 Avenue Centrale, 38401 Saint-Martin-d'H\`eres}
\email{bozhidar.velichkov@univ-grenoble-alpes.fr}

\makeindex

\begin{document}


\begin{abstract}
Using a direct approach, we prove a $2$-dimensional epiperimetric inequality for the one-phase problem in the scalar and vectorial cases and for the double-phase problem. From this we deduce, in dimension $2$, the $C^{1,\alpha}$ regularity of the free-boundary in the scalar one-phase and double-phase problems, and of the reduced free boundary in the vectorial case, without any restriction on the sign of the component functions. Furthermore we show that in the vectorial case the free boundary can end in a cusp.
\end{abstract}

\maketitle



\section{Introduction}
In this paper we prove the $C^{1,\alpha}$ regularity of the free-boundary in the two dimensional case for the following three problems: the classical one-phase problem in the scalar and vectorial case, and the double-phase problem. In the scalar case, the regularity of the free boundary at flat points was first proved by Alt-Caffarelli for the one-phase (see \cite{AlCa}) and by Alt-Caffarelli-Friedman for the double phase (see \cite{AlCaFr}). In the vectorial case, the main results are due to Caffarelli-Shahgholian-Yeressian, Kriventsov-Lin and Mazzoleni-Terracini-Velichkov (see \cite{CaShYe, KrLi, MaTeVe}), and although they hold in any dimension, they all require additional assumptions on the positivity of the components of the vector valued minimizer. 
While these results rely on the so called improved flatness, based on a boundary Harnack principle, our result is achieved by proving first an epiperimetric inequality for the boundary datum of a minimizer, and then applying standard techniques first introduced in the context of free boundary problems by Weiss, Focardi-Spadaro and Garofalo-Petrosyan-Garcia (cp. \cite{Weiss2,FoSp,GaPeVe}). Notice also that, differently than these results, our proof of the epiperimetric inequality is direct: that is we construct an explicit competitor whose energy is strictly smaller than the $1$-homogeneous extension of the boundary datum. The idea for such a construction is inspired by results of Reifenberg and White (cp. \cite{Reif2, Wh}) in the context of minimal surfaces, which need to be substantially modified to take into account the measure term which appears in our functional. 

Finally we remark that our work is inspired by Weiss' observation in the context of the obstacle problem which states "...it should however be possible to give a direct proof of the epiperimetric inequality which would then also cover singular sets of
intermediate dimension" (see \cite{Weiss2}). Indeed, in forthcoming work joint with Max Engelstein, we will extend our results to dimension higher than $2$ and use it to study some special singular points of the free-boundary.  

\subsection{Statements of the main Theorems}

Let $\Omega\subset \R^d$ be an open set and consider the following three functionals:
\begin{itemize}
\item[(OP)] $\ds \mathcal{E}_{OP}(u):=\int_\Omega |\nabla u|^2\,dx+\big|\{u>0\}\cap\Omega\big|$, where $u\ge 0$ and $u\in H^1(\Omega)$;
\item[(DP)] $\ds \mathcal{E}_{DP}(u):=\int_\Omega |\nabla u|^2\,dx+\lambda_1\,\big|\{u>0\}\cap\Omega\big|+\lambda_2\,\big|\{u<0\}\cap \Omega\big|$, where $u\in H^1(\Omega)$ and $\lambda_1,\lambda_2>0$, $\lambda_1\neq \lambda_2$;
\item[(V)] $\ds \mathcal{E}_{V}(u):=\int_\Omega |\nabla u|^2\,dx+\big|\{|u|>0\}\cap\Omega\big|=\sum_{i=1}^n\int_\Omega|\nabla u_i|^2\,dx+\big|\{|u|>0\}\cap\Omega\big|$, where $u\in H^1(\Omega; \R^n)$ and $n\ge 1$.
\end{itemize}
We say that
\begin{itemize}
\item $u\in H^1(\Omega)$ is a \emph{minimizer of $\mathcal{E}_{OP}$} in $\Omega$, if $u\ge0$ and $\mathcal{E}_{OP}(u)\leq \mathcal{E}_{OP}(\tilde u)$ for every $\tilde u\in H^1(\Omega)$ with $\tilde u|_{\de \Omega}=u|_{\de \Omega}$, that is $u-\tilde u\in H^1_0(\Omega)$;
\item  $u\in H^1(\Omega)$ is a \emph{minimizer of $\mathcal{E}_{DP}$} in $\Omega$, if $\mathcal{E}_{DP}(u)\leq \mathcal{E}_{DP}(\tilde u)$ for every $\tilde u\in H^1(\Omega)$ with $u-\tilde u\in H^1_0(\Omega)$;
\item  $u\in H^1(\Omega;\R^n)$ is a \emph{minimizer of $\mathcal{E}_{V}$} in $\Omega$, if $\mathcal{E}_{V}(u)\leq \mathcal{E}_{V}(\tilde u)$ for every $\tilde u\in H^1(\Omega;\R^n)$ with $u-\tilde u\in H^1_0(\Omega;\R^n)$.
\end{itemize}
Since many results and notions are common for minimizers of $\mathcal{E}_{OP}$, $\mathcal{E}_{DP}$ and $\mathcal{E}_{V}$, from now on we will often replace the indices $OP$, $DP$ and $V$ by $\square$. When $\square=V$ we will assume that the arguments are $\R^n$-valued functions, where $n\ge 1$ is a fixed integer.

For $r>0$, $x_0\in\R^d$ and $u\in H^1(B_r(x_0);\R^n)$ we define the functional $W_0$  by
$$
W_0(u,r,x_0):=\frac{1}{r^d}\int_{B_r(x_0)}|\nabla u|^2\,dx-\frac1{r^{d+1}}\int_{\partial B_r}|u|^2\,d\HH^{d-1}.
$$ 
The \emph{Weiss' boundary-adjusted energy}, associated to $\mathcal E_{OP}$, $\mathcal E_{DP}$ and $\mathcal E_{V}$, is given by
\begin{gather}\label{weiss}
W^{OP}(u,r,x_0)=W_0(u,r,x_0)+\frac{1}{r^d}\big|\{u>0\}\cap B_r(x_0)\big|\\
W^{DP}(u,r,x_0)=W_0(u,r,x_0)+\frac{1}{r^d}\Bigl(\lambda_1\,\big|\{u>0\}\cap B_r(x_0)\big|+\lambda_2\,\big|\{u<0\}\cap B_r(x_0)\big|\Bigr)\\
W^{V}(u,r,x_0)=W_0(u,r,x_0)+\frac{1}{r^d}\big|\{|u|>0\}\cap B_r(x_0)\big|\,.
\end{gather}

By a celebrated result of Weiss (see \cite{Weiss1} for the scalar case and \cite{CaShYe, MaTeVe} for the vectorial case), these functionals are monotone non-decreasing in $r$. In particular, there exists the \emph{density of $u$ at $x_0$} defined as
$$\Theta^\square_u(x_0):=W^\square(u,0,x_0)=\lim_{r\to 0}W^\square(u,r,x_0).$$
Thanks to results of Alt-Caffarelli and Alt-Caffarelli-Friedman (see \cite{AlCa,AlCaFr}), later refined by Caffarelli-Jerison-Kenig and Jerison-Savin (see \cite{CaJeKe,JeSa}), we have that
\begin{itemize}
\item[(OP)] if $2\leq d\leq 4$ and $x_0\in \de \{u>0\}$, then $\ds\Theta^{OP}_u(x_0)=\frac{\omega_d}2$, where $\omega_d$ is the volume of the unit ball in dimension $d$;
\item[(DP)] if $2\leq d \leq 4$, then $x_0\in \de \{u>0\}\setminus \de \{u<0\}$ implies $\ds
\Theta_u^{DP}(x_0)=\lambda_1\frac{\omega_d}2$,

\hspace{2.7cm}$x_0\in \de \{u<0\}\setminus \de \{u>0\}$ implies $\ds\Theta^{DP}_u(x_0)=\lambda_2\frac{\omega_d}2$,

\hspace{2.6cm} $x_0\in \de \{u>0\}\cap \de \{u<0\}$ implies $\ds\Theta^{DP}_u(x_0)=(\lambda_1+\lambda_2)\frac{\omega_d}2$;
\item[(V)] if $d=2$ and $x_0\in \de \{|u|>0\}$, then either $\ds\Theta^{V}_u(x_0)=\frac\pi2$ or $\Theta^{V}_u(x_0)=\pi$. In particular, if $x_0\in \de_{red} \{|u|>0\}$, then $\ds\Theta^{V}_u(x_0)=\frac\pi2$ (see Subsection \ref{ss:class_bu}).  
\end{itemize}

The epiperimetric inequality improves the monotonicity of $W^\square$ to a rate of convergence to the density $\Theta^\square$. Since $W^\square$ has the scaling property 
$$W^\square(u,r,x_0)=W^\square(u_r,1,0),\quad\text{where} \quad u_r(x)=\frac1r u(x_0+rx),$$
we can suppose that $x_0=0$ and $r=1$ and for  the sake of simplicity we set $W^\square(u):=W^\square(u,1,x_0)$.

For the one-phase problem in the scalar case we have the following result. 

{
\begin{theo}[Scalar epiperimetric inequality for the one-phase problem]\label{p:epi_1}
	For every $\alpha>0$ there is $\eps>0$ such that if $c\in H^1(\partial B_1)$ is a non-negative function satisfying $\ds\int_{\partial B_1} c>\alpha$, then there is a function $h\in H^1(B_1)$ such that $h=c$ on $\de B_1$ and
	\begin{equation}\label{e:epi_1}
	W^{OP}(h)-\frac\pi2\leq (1-\eps)\left(W^{OP}(z)-\frac\pi2\right)\,,
	\end{equation}
	where $z\in H^1(B_1)$ is the one-homogeneous extension of the trace of $c$ to $B_1$.
\end{theo}}

\noindent In the cases $\square=DP$ and $\square=V$ the functional $W^\square(u,x_0,r)$ behaves differently in  points $x_0$ of the free boundary with different densities. We distinguish two cases. 
\begin{enumerate}[-]
	\item The high density points $x_0$, that is the points $x_0$ such that $\ds\Theta^{DP}_u(x_0)=(\lambda_1+\lambda_2)\frac{\pi}2$ or $\Theta^{V}_u(x_0)=\pi$. For the minimizers of $\mathcal F_{DP}$ these are precisely the points of the double-phase boundary. In the case of $\mathcal F_V$ there are several possibilities: the high density points can be isolated, double-phase points or they might be the vertex of an entering cusp. In all these cases the epiperimetric inequality holds at all scales. 
	\item The points of low density,  that is the points $x_0$ such that $\ds\Theta^{DP}_u(x_0)=\lambda_1\frac{\pi}2$, $\ds\Theta^{DP}_u(x_0)=\lambda_2\frac{\pi}2$ or $\ds\Theta^{V}_u(x_0)=\frac\pi2$. In the case of $\mathcal F_{DP}$, these are the points of the one-phase boundaries $\partial \{u>0\}\setminus \partial\{u<0\}$ et $\partial \{u<0\}\setminus \partial\{u>0\}$. In the case of $\mathcal F_{V}$, the points of low density are precisely the points of the reduced free boundary $\partial_{red}\{|u|>0\}$. In these cases the epiperimetric inequality holds only starting from a sufficiently small radius depending on the point $x_0$.
\end{enumerate}

\noindent The precise statements are the following.

{
\begin{theo}[Scalar epiperimetric inequality for the double-phase problem]\label{p:epi_2}
For every $\alpha>0$ there is $\eps>0$ such that
 if $c\in H^1(\partial B_1)$ is a function satisfying  $\ds\int_{\partial B_1} c^+>\alpha$ and $\ds\int_{\partial B_1} c^->\alpha$, then there is a function $h\in H^1(B_1)$ such that $h=c$ on $\de B_1$ and
\begin{equation}\label{e:epi_13}
W^{DP}(h)-(\lambda_1+\lambda_2)\frac{\pi}2\leq (1-\eps)\left(W^{DP}(z)-(\lambda_1+\lambda_2)\frac{\pi}2\right)\,,
\end{equation}
where $z\in H^1(B_1)$ is the one-homogeneous extension of $c$ to $B_1$. 
\end{theo}}

{
\begin{theo}[Epiperimetric inequality for vector-valued functions]\label{p:epi_vec}
Let $n\ge 1$ and $B_1\subset\R^2$. For every $\delta_0>0$ there is $\eps>0$ such that 
\begin{itemize}
\item[(i)] if $c\in H^1(\partial B_1;\R^n)$ and $\ds\big|\{|c|>0\}\cap \de B_1\big|\leq 2\pi-\delta_0$, then there is a function $h\in H^1(B_1;\R^n)$ such that $h=c$ on $\de B_1$ and
\begin{equation}\label{e:epi_vec_low_density}
W^V(h)-\frac\pi2\leq (1-\eps)\left(W^V(z)-\frac\pi2\right)\,;
\end{equation}
\item[(ii)] if $c\in H^1(\partial B_1;\R^n)$, then there is a function $h\in H^1(B_1;\R^n)$ such that $h=c$ on $\de B_1$ and
\begin{equation}\label{e:epi_vec}
W^V(h)-\pi\leq (1-\eps)\,(W^V(z)-\pi)\,.
\end{equation}
\end{itemize}
In both cases $z\in H^1(B_1;\R^n)$ is the one-homogeneous extension of $c$ in $B_1$.
\end{theo}}

As a consequence of the epiperimetric inequalities we obtain the uniqueness of the blow-up limits and the regularity of the free boundary following a standard procedure (see \cite{FoSp}). For the next theorem we recall the standard notation $ u_{r,x_0}(x):=\frac1ru(x_0+rx)$ and $u_r(x):=u_{r,0}(x)$.

\begin{theo}[Uniqueness of the blow-up limits]\label{t:uniq_blowup} Let $\Omega\subset \R^2$ be a given open set.
\begin{enumerate}[(i)]
\item[(OP)] Suppose that $u\in H^1(\Omega)$ is a minimizer of $\mathcal E_{OP}$ in $\Omega$ and $x_0\in \partial\{u>0\}\cap\Omega$. Then there is a unit vector $e=e_{x_0}\in\partial B_1$ such that $u_r$ converges,  as $r\to 0$, to the function $h(x):=\max\{0,e\cdot x\}$ locally uniformly and in $H^1_{loc}(\R^2)$.
\item[(DP)] Suppose that $u\in H^1(\Omega)$ is a minimizer of $\mathcal E_{DP}$ in $\Omega$ and $x_0\in \partial\{|u|>0\}\cap\Omega$. Then there is a unit vector $e=e_{x_0}\in\partial B_1$ such that $u_r$ converges locally uniformly and in $H^1_{loc}(\R^2)$ to the function $h$, defined as: 
\begin{equation*}
\begin{array}{rl}
h(x)=\lambda_1\max\{0,e\cdot x\},&\text{if}\quad x_0\in\partial\{u>0\}\setminus\partial\{u<0\};\\
h(x)=\lambda_2\min\{0,e\cdot x\},&\text{if}\quad x_0\in\partial\{u<0\}\setminus\partial\{u>0\};\\
h(x)=\mu_1\max\{0,e\cdot x\}+\mu_2\min\{0,e\cdot x\},&\text{if}\quad x_0\in\partial\{u>0\}\cap\partial\{u<0\},
\end{array}
\end{equation*}
where $\mu_1\ge\lambda_1$, $\mu_2\ge\lambda_2$ and $\mu_1^2-\mu_2^2=\lambda_1^2-\lambda_2^2$.
\item[(V)] Suppose that $u\in H^1(\Omega;\R^n)$ is a minimizer of $\mathcal E_{V}$ in $\Omega$ and $x_0\in \partial\{|u|>0\}\cap\Omega$. Then there is a function $h:\R^2\to\R^n$ such that $u_r$ converges locally uniformly and in $H^1_{loc}(\R^2;\R^n)$ to $h$, where 
\begin{itemize}
\item if $\Theta^V_u(x_0)={\pi}/2$, then $h(x)=\xi\max\{0,e\cdot x\}$, where $\xi\in\R^n$, $|\xi|=1$, $e\in\mathbb{S}^1$;
\item if $\Theta^V_u(x_0)=\pi$, then $h(x)=(\xi^1\,h_{e_1}(x),\dots, \xi^n\,h_{e_n}(x))$, where $h_e(x)=e\cdot x$, with $e\in\mathbb{S}^1$, and $\xi^i\in \R$. Moreover, if there exists $i\neq j\in \{1,\dots,n\}$ such that $e_i\neq e_j$ then $x_0$ is isolated in $\de\{|u|>0\}\cap \Omega$. In particular, if $x_0$ is not isolated in $\de\{|u|>0\}\cap \Omega$ and $\Theta^V_u(x_0)=\pi$, then the blow-up in $x_0$ is of the form $h=\xi\,h_e$, for some $e\in \mathbb{S}^1$ and $\xi \in \R^n$.
\end{itemize}
\end{enumerate}
\end{theo}

\begin{theo}[Regularity of the free boundary]\label{t:main}
Let $\Omega\subset\R^2$ be an open set. There exists a universal constant $\alpha>0$ such that:\begin{enumerate}
\item[(OP)] if $u\in H^1(\Omega)$ is a minimizer of $\mathcal E_{OP}$ in $\Omega$, then $\de\{u>0\}\cap \Omega$ is locally a graph of a $C^{1,\alpha}$ function;
\item[(DP)] if $u\in H^1(\Omega)$ is a minimizer of $\mathcal E_{DP}$ in $\Omega$, then both $\de\{u>0\}\cap\Omega$ and $\de\{u<0\}\cap\Omega$ are locally graphs of $C^{1,\alpha}$ functions;
\item[(V)] if $u\in H^1(\Omega;\R^n)$ is a minimizer of $\mathcal E_{D}$ in $\Omega$, then the reduced free boundary $\de_{red}\,\{|u|>0\}\cap \Omega$ is locally a graph of a $C^{1,\alpha}$ function. 
\end{enumerate}
\end{theo}



\noindent It is important to notice that the free boundary in the vectorial case may in a cusp, indeed we have the following example.

\begin{exam}\label{ex:cusp_sing}
There exists a local minimizer $u:\R^2\to\R^2$ of the functional $\mathcal E_V$ for which 
\begin{enumerate}[(1)]
\item $\Omega_u=\{|u|>0\}$ is a connencted open set; 
\item there is a point $x_0\in\partial\Omega_u$ of density $\Theta_u^V(x_0)=\pi$. 
\end{enumerate} 
This is a completely different behavior with respect to the one-phase and double-phase problems. For the one-phase problem the points of density $\omega_d$ are not admitted in any dimension. On the other hand, for the double-phase problem, if the point $x_0\in\partial\{|u|>0\}$ is of density $\pi$, then the two sets $\{u>0\}$ and $\{u<0\}$ meet in $x_0$ and they are both $C^{1,\alpha}$ regular.   
\end{exam}

\noindent Finally we remark that Theorems \ref{t:uniq_blowup} and \ref{t:main} remain true if we replace the measure terms in our functionals by a H\"older continuous weight function  $q:\Omega\to\R^+$, that is we define 
\begin{gather}
\mathcal{E}^q_{OP}(u):=\int_\Omega \left[|\nabla u|^2+q(x)\,\chi_{\{u>0\}}\right]\,dx , \notag \\
\mathcal{E}^q_{DP}(u):=\int_\Omega \left[ |\nabla u|^2+q_1(x)\,\chi_{\{u>0\}}+q_2(x)\,\chi_{\{u<0\}}\right]\,dx ,\notag \\
\mathcal{E}^q_{V}(u):=\int_\Omega \left[|\nabla u|^2+q(x)\,\chi_{\{|u|>0\}}\right]\,dx ,\notag
\end{gather}
where $\chi_A$ denotes the characteristic function of a set $A$.
The minimizers of these functionals are in fact almost minimizers of the original functionals $\mathcal{E}_\square$, so that we can prove the following

\begin{theo}[H\"older continuous weight functions]\label{c:Q_funct}
Let $\Omega\subset\R^2$ be an open set and  $q,q_1,q_2\in C^{0,\gamma}(\Omega;\R^+)$ be H\"older continuous functions such that $q, q_1, q_2\ge c_q>0$, where $c_q$ is a given constant. There exists a constant $\alpha>0$ such that:
\begin{enumerate}
\item[(OP)] if $u\in H^1(\Omega)$, $u\ge 0$, is a minimizer of $\mathcal E_{OP}^q$ in $\Omega$, then $\de\{u>0\}\cap \Omega$ is locally a graph of a $C^{1,\alpha}$ function;
\item[(DP)] if $u\in H^1(\Omega)$ is a minimizer of $\mathcal E_{DP}^q$ in $\Omega$, then $\de\{u>0\}\cap\de \{u<0\}\cap\Omega$ is locally a closed subset of a graph of a $C^{1,\alpha}$ function;
\item[(V)] if $u\in H^1(\Omega;\R^n)$ is a minimizer of $\mathcal E_{D}^q$ in $\Omega$, then the reduced free boundary $\de_{red}\,\{|u|>0\}\cap \Omega$ is locally a graph of a $C^{1,\alpha}$ function. 
\end{enumerate}
Moreover, the blow-up limits of the minimizers of $\mathcal E_{\square}^q$ are unique and are given precisely by the classification in Theorem \ref{t:uniq_blowup}. 
\end{theo}

\subsection{Sketch of the proof of Theorem \ref{p:epi_1}}

Since the epiperimetric inequality is the key and new part of our work, we sketch its proof here in the case $\square=OP$, the other cases being similar.

\noindent Given $u\in H^1(B_1)\cap C^0(B_1)$ as in the statement of Theorem \ref{p:epi_1} we consider the trace $c:=u|_{\de B_1}$ and its positivity set $S:=\{c>0\}\subset\partial B_1$. 

\noindent We first show that there exists a dimensional constant $\delta_0>0$ such that, if $|S|\geq 2\pi-\delta_0$, then the harmonic extension of $c$ in the ball $B_1$ satisfies \eqref{e:epi_1}. Loosely speaking this means that, in the regime where the positivity set $\{z>0\}\cap B_1$ is almost the whole ball, the energy gain is bigger than any loss in measure (cp. Subsection \ref{ss:Sbbig}).
	
\noindent Next we assume that $|S|\le 2\pi -\delta_0$; a natural candidate for the function $h$ is the continuous function $\tilde h:B_1\to\R$ such that: 
\begin{itemize}
	\item $\tilde h$ is harmonic on the cone $\bC_S$ generated by the support $S$ of the boundary datum $c$
	\begin{equation}\label{e:CS}\ds \bC_S=\Big\{\lambda\theta\ :\ \lambda\in ]0,1[\,,\ \theta\in S\Big\} ;\end{equation}
	\item $\tilde h=c$ on $\partial B_1$ and $\tilde h=0$  outside $\bC_S$.
\end{itemize}
This function provides an immediate improvement of the term $W_0$ (we deal with the decomposition of $\tilde h$ in Fourier series and the subsequent energy estimates in Subsection \ref{ss:tildeh}), but it does not take into account the measure term in $W^\square$. In order to deal with it, we have to modify $\tilde h$ by appropriately adding measure or cutting off pieces from the cone $\bC_S$. To do this we divide the support $S=\{c>0\}\subset\partial B_1$ into disjoint sets $S=S_{big}\cup \bigcup_i S^i_{small}$, according to the parameter $\delta_0$, in the following way :

$\bullet$ $S^i_{small}$ are the connected components of $S$ whose measure does not exceed $ \ds \pi-\frac{\delta_0}{4}$; 

$\bullet$ $S_{big}=S\setminus \bigcup_i S^i_{small}$.
 Notice that in general $S_{big}$ could be the empty set, but if not, then it is connected and $\ds \pi-\frac{\delta_0}{4}\le |S_{big}|\le 2\pi-\delta_0$. In fact, if $S_{big}$ had two or more connected components, then the measure of $S$ would exceed $2\pi-\delta_0$.  

 \noindent We modify the function $\tilde h$ on $S_{small}$ by  a truncation argument with a suitably chosen cut-off function supported in a small ball centered in the origin. Since we use this truncation in other parts of the paper, the main estimate is proved separately in Subsection \ref{ss:cut}. Roughly speaking, this improves $W^\square$ because the first eigenvalue of $S_{small}$ is a dimensional constant bigger than $(d-1)$, that is we are far away from the half sphere, which is the linear solution. 

\noindent In order to construct an appropriate competitor on $S_{big}$, we represent the restriction $c|_{S_{big}}$ as 
$$c(\theta)=c_1\phi_1(\theta)+g(\theta),$$
where $c_1$ is a constant, $\phi_1$ is the first eigenfunction on $S_{big}$ and $g$ contains all the higher frequencies of $c$. 
For the higher frequencies $g$, the usual harmonic extension combined with the same cut-off argument used for $S_{small}$, gives the required improvement (this is once again because the second eigenvalue on $S_{big}$ is bigger than $(d-1)$ plus a geometric constant). It is interesting to notice that, up to this point, the argument works in every dimension. For the first frequency $c_1 \, \phi_1$, we use an internal variation, supported in the ball cut-off from the higher frequencies, to move the support of $\phi_1$ in the direction of the half plane solution $\max\{0,e \cdot x\}$, whose trace is given precisely by $\phi_1$. The improvement on $S_{big}$ is contained in Subsections \ref{ss:big} and \ref{ss:big_cones_double} for the one and double phase respectively.  \\
\qed

\subsection{Organization of the paper} 
The rest of the paper is divided into three sections. In Section \ref{s:pre} we recall some basic properties of the minimizers of the functionals $\mathcal E_{OP}$, $\mathcal E_{DP}$ and $\mathcal E_{VP}$, and do some preliminary standard computations related to harmonic extensions and the cut-off function we use. In Section \ref{s:epi} we prove the epiperimetric inequalities of Theorems \ref{p:epi_1}, \ref{p:epi_2} and \ref{p:epi_vec}, while the last section is dedicated to the proofs of Theorems \ref{t:uniq_blowup}, \ref{t:main} and \ref{c:Q_funct}.

\subsection{Acknowledgements} The authors are grateful to Emanuele Spadaro and Guido De Philippis for many suggestions and interesting conversations.

\section{Preliminary results and computations}\label{s:pre}

In this section we recall some regularity results for local minimizers of $\mathcal E_{\square}$ and we carry out some preliminary computations that will be useful in the sequel. Many times we will drop the index $\square$, when it will be clear from the context which functional we are referring to.

\subsection{Non-degeneracy and Lipschitz regularity}\label{ss:preliminaries} 

In this section we recall some well-known results about the one-phase and double-phase problems, that is the Lipschitz continuity and the non-degeneracy of the minimizers. 

\begin{lemma}[Regularity and non-degeneracy of local minimizers of (OP) and (DP)]\label{e:properties_of_min}
Let $\Omega\subset\R^d$ be an open set, $q,q_1,q_2\in C^{0,\gamma}(\Omega;\R^+)$ H\"older continuous functions such that 
$q, q_1, q_2\ge 1$, and $u\in H^1(\Omega)$ be a minimizer of either $\mathcal{E}^q_{OP}$ or $\mathcal E^q_{DP}$. Then the following properties hold:
\begin{itemize}
\item[(i)] $u\in C^{0,1}_{loc}(\Omega)$.
\item[(ii)] There is a dimensional constant $\alpha>0$ such that for every $x_0\in \de \{u^{\pm}>0\}\cap \Omega$ and every $0<r<\text{dist}(x_0,\partial\Omega)$ we have $\ds\int_{\de B_r(x_0)}u^\pm\geq \alpha\,r$, where we note $u^\pm=\max\{\pm \,u,0\}$.
\end{itemize}
\end{lemma}

\begin{proof}
For the one-phase functional $\mathcal E_{OP}$, the first property follows from \cite[3.3 Corollary]{AlCa}, while the second follows from \cite[3.4 Lemma]{AlCa}. For the double-phase problem $\mathcal E_{DP}$, (ii) is the content of \cite[Theorem 3.1]{AlCaFr}, while (i) for $u^+$ is proved in \cite[Theorem 5.3]{AlCaFr}, and the proof for $u^-$ is exactly the same. More general proofs of (i), valid in both our situations, are given in \cite{CaJeKe}, where the authors extend it to the inhomogeneous case, or in  \cite{BuMaPrVe}, where the point of view of almost minimization is used.
\end{proof}

A similar statement is true for the vectorial case (see \cite{MaTeVe,CaShYe}).

\begin{lemma}[Regularity and non-degeneracy of vector-valued minimizers]\label{l:properties_of_min_vect}
Let $\Omega\subset\R^d$ be an open set, $q\in C^{0,\gamma}(\Omega;\R^+)$ a H\"older continuous function such that 
$q\ge 1$, and $u\in H^1(\Omega;\R^n)$ be a minimizer of $\mathcal{E}^q_V$. Then 
\begin{itemize}
\item[(i)] $u\in C^{0,1}_{loc}(\Omega;\R^n)$. 
\item[(ii)] There is a dimensional constant $\alpha>0$ such that for every $x_0\in \de \{|u|>0\}\cap \Omega$ and every $0<r<\text{dist}(x_0,\partial\Omega)$ we have 
$\ds\int_{\de B_r(x_0)}|u|\geq \alpha \,r\,.$
\end{itemize}
\end{lemma}

\begin{proof} The proof of (ii) can be found in \cite[Lemma 2.9]{MaTeVe}, while for (i) we make the following observation. Let $i\in\{1,\dots,n\}$ and $\phi\in C^\infty_c(B_r(x_0))$, $B_r(x_0)\subset \Omega$, then for some constant $C>0$ the following inequality holds
$$ \int_{B_r(x_0)}|\nabla u_i|^2\,dx\leq \int_{B_r(x_0)}|\nabla (u_i+\phi)|^2\,dx+ C \,r^d\,,  $$
that is each component of $u=(u_1,\dots,u_n)$ is a quasi-minimizer for the Dirichlet energy and is harmonic where it is not zero (since $u$ is a minimizer of $\mathcal{E}_V$). The result then follows by \cite[Theorem 3.3]{BuMaPrVe}.
\end{proof}

\begin{oss}
We remark that the Lipschitzianity of the solutions to all of our problems is indeed equivalent to the fact that the components of the solutions are quasi-minimizer for the Dirichlet energy as described in the proof of Lemma \ref{l:properties_of_min_vect} (see \cite{BuMaPrVe}).
\end{oss}

\subsection{Classification of blow-ups in the vectorial case}\label{ss:class_bu}

The possible blow-up limits for the one-phase and the double-phase problems are well-known in dimension two. For the sake of completeness, we prove in this section the classification of the possible blow-ups in the vectorial case. The precise statement is the following.

\begin{lemma}\label{l:class_bu_V} If $h\in H^1(B_1(x_0);\R^n)$, $B_1\subset \R^2$, arises as the blow-up of a minimizer $u$ to the functional $\mathcal{E}_V$ at a free boundary point $x_0\in \de \{|u|>0\}$, that is there exists a subsequence $(u_{r_k})_k$ of $(u_r)_r$ which converges to $h$, then we have two possibilities  
	\begin{itemize}
		\item $\Theta^V_u(x_0)={\pi}/2$ and $h(x)=\xi\max\{0,e\cdot x\}$, where $\xi\in\R^n$, $|\xi|=1$, $e\in\mathbb{S}^1$;
		\item $\Theta^V_u(x_0)=\pi$ and $h(x)=(\xi^1\,h_{e_1}(x),\dots, \xi^n\,h_{e_n}(x))$, where $h_e(x)=e\cdot x$, with $e\in\mathbb{S}^1$, and $\xi^i\in \R$. 
	\end{itemize}
\end{lemma} 

\begin{proof}
Assume that $x_0=0$. We start by noticing that by standard argument and the Weiss' monotonicity formula, $h$ is a $1$-homogeneous minimizer of $\mathcal{E}_V$ and each component is harmonic on the cone $\{|h|>0\}\cap B_1$, see for instance \cite{CaShYe, MaTeVe}. Then we have two possibilities. 
\begin{itemize}
	\item  $\{|h|>0\}=\{e\cdot x>0\}$, in which case $h(x)=\xi\max\{0,e\cdot x\}=:\xi h_e(x)$ and $\Theta^V_u(x_0)={\pi}/2$. Moreover, for any function $\phi\in H^1(B_1)$, consider the competitor $\xi \, \phi$, then
$$
|\xi|^2\int_{B_1}|\nabla  h_e|^2+|\{|h_e|>0\}|\leq |\xi|^2\int_{B_1}|\nabla  (h_e+\phi)|^2+|\{|h_e+\phi|>0\}|\,,
$$
that is $h_e$ minimizes the functional $\ds \int_{B_1}|\nabla  h_e|^2+\frac{1}{|\xi|^2}|\{|h_e|>0\}|$, which by the classification of the $1$-homogeneous solutions to the scalar one-phase problem, implies that $|\xi|=1$.
	\item $|\{|h|>0\}|=\pi$, in which case all the components of $h$ are harmonic functions in $B_1$. Indeed assume without loss of generality that the first component of the blow up $h^1$ is not harmonic, then it is easy to see that, if $\bar h^1$ is the harmonic extension of the trace of $h^1$, then
	$$
	\int_{B_1}|\nabla \bar h^1|^2+\sum_{i=2}^n\int_{B_1}|\nabla h^i|^2+\pi <  \int_{B_1}|\nabla h^1|^2+\sum_{i=2}^n\int_{B_1}|\nabla h^i|^2+\pi\,,
	$$
	which is a contradiction with the minimality of the blow up $h$. By the $1$-homogeneity of $h$ all the functions are linear, which concludes the proof.
\end{itemize}

\end{proof}

\subsection{Harmonic extension of the boundary datum}\label{ss:tildeh}
Let $S$ be an open subset of the unit sphere $\partial B_1.$
On $S$ we consider the sequence of Dirichlet eigenfunctions $\phi_j$, $j\ge1$, and the corresponding eigenvalues $\lambda_j$, $j\ge 1$, counted with their multiplicity on the spherical subset $S$. We have that each $\phi_j$ solves the PDE
$$-\Delta_S \phi_j=\lambda_j\phi_j\quad\text{in}\quad S,\qquad \phi_j=0\quad\text{on}\quad \partial S,\qquad\int_{S}\phi_j^2(\theta)\,d\theta=1,$$
where $\Delta_S$ denotes the Laplace-Beltrami operator on the unit sphere $\partial B_1$ and $\theta$ is the variable on $S$. Given a Sobolev function $c\in H^1_0(S;\R^n)$ on the sphere, we set 
$$\ds \R^n\ni c_j:=\int_{\de B_1} c(\theta)\,\phi_j(\theta)\,d\theta.$$ 
Then we can express $c$ as a Fourier series 
$$c(\theta)=\sum_{j=k}^{\infty}c_j \phi_j(\theta)\,,\qquad c_j\in \R^n\mbox{ for every }j\geq k$$
converging in $H^1(S; \R^n)$, where $k\in \N$ is the first value for which $c_k\neq 0$. We consider the radial and the harmonic extensions, $z$ and $\tilde h$, of $c$ inside the cone $\bC_S$ defined in \eqref{e:CS}.
In polar coordinates $z$ and $\tilde h$ are given by  
\begin{equation}\label{e:har&rad}
z(r,\theta)=\sum_{j=k}^\infty r\,c_j \, \phi_j(\theta)\qquad\text{and}\qquad \tilde{h}(r,\theta)=\sum_{j=k}^\infty  r^{\alpha_j}\,c_j\,\phi_j(\theta),
\end{equation}
where $\alpha_j=\alpha_j(S)$ is the homogeneity of the harmonic extension of $\pi_j$ on $\bC_S$ which also can be defined through the identity 
$$\alpha_j(\alpha_j+d-2)=\lambda_j\quad\text{for every}\quad j\in\N.$$

\begin{lemma}[Harmonic extension]\label{l:harm_ext}
	Let $S\subset \partial B_1$ be an open subset of the unit sphere and $c$, $\tilde h$ and $z$ as above. For every $\eps\in[0,1]$ we have
	\begin{gather}\label{e:ringhio}
	W_0(\tilde{h})=\sum_{j=k}^\infty |c_j|^2\left(\alpha_j-1\right),\qquad W_0(z)=\sum_{j=k}^\infty |c_j|^2\left(\frac{1+\lambda_j}d-1\right),\\
	\label{e:pjanic}
	W_0(\tilde{h})-(1-\eps)W_0(z)=-\frac{1-\eps}{d}\sum_{j=k}^\infty |c_j|^2(\alpha_j-1)\Big(\alpha_j-1+d-\frac{d}{1-\eps}\Big).
	\end{gather}
	In particular, if $\alpha_k>1$, then 
	\begin{equation}\label{e:dybala}
	\text{for every}\quad \eps\le\frac{\alpha_k-1}{d+\alpha_k-1},\quad\text{we have}\quad W_0(\tilde{h})-(1-\eps)W_0(z)\le 0.
	\end{equation}
\end{lemma}
\begin{proof}
	We first calculate $W_0(z)$ and $W_0(\tilde{h})$. By the orthogonality of $\phi_j$ in $H^1(S)$, that is 
	$$\int_{S}\phi_i\phi_j\,d\theta=\delta_{ij},\qquad \int_{S}\nabla_\theta\phi_i\cdot\nabla_\theta\phi_j\,d\theta=\lambda_i\delta_{ij},$$
	we have
	\begin{align*}
	\int_{B_1}|\nabla \tilde{h}|^2\,dx&= \int_{0}^1\int_{S}\left(|\partial_r \tilde{h}(r,\theta)|^2+\frac{|\nabla_\theta \tilde{h}(r,\theta)|^2}{r^2}\right)r^{d-1}\,d\theta\,dr\\
	&=\sum_{j=k}^\infty \int_{0}^1\int_{S}|c_j|^2\left(\alpha_j^2 r^{2(\alpha_j-1)}\phi_j^2(\theta)+\frac{r^{2\alpha_j}|\nabla_\theta \phi_j(\theta)|^2}{r^2}\right)r^{d-1}\,d\theta\,dr\\
	&=\sum_{j=k}^\infty \int_{0}^1|c_j|^2\left(\alpha_j^2 +\lambda_j\right)r^{2(\alpha_j-1)+d-1}\,dr=\sum_{j=k}^\infty \frac{|c_j|^2(\alpha_j^2 +\lambda_j)}{2\alpha_j+d-2}=\sum_{j=k}^\infty |c_j|^2\alpha_j.
	\end{align*}
	\begin{align}
	\int_{B_1}|\nabla z|^2\,dx=&\int_{0}^1\int_{S}\left(|\partial_r z(r,\theta)|^2+\frac{|\nabla_\theta z(r,\theta)|^2}{r^2}\right)r^{d-1}\,d\theta\,dr\notag\\
	&=\sum_{j=k}^\infty \int_{0}^1\int_{S}|c_j|^2\left(\phi_j^2(\theta)+\frac{r^2|\nabla_\theta \phi_j(\theta)|^2}{r^2}\right)r^{d-1}\,d\theta\,dr\notag\\
	&=\sum_{j=k}^\infty \int_{0}^1|c_j|^2\left(1+\lambda_j\right)r^{d-1}\,dr=\frac1d\sum_{j=k}^\infty |c_j|^2(1+\lambda_j),\label{e:energyz}
	\end{align}
	$$\int_{\partial B_1}|z|^2\,d\HH^{d-1}=\int_{\partial B_1}|\tilde{h}|^2\,d\HH^{d-1}=\int_S|c|^2(\theta)\,d\theta=\sum_{j=k}^\infty |c_j|^2.$$
	Now for any $\eps\in]0,1[$ we get
	\begin{align*}
	W_0(\tilde{h})-(1-\eps)W_0(z)
	&=\sum_{j=k}^\infty |c_j|^2\left[\alpha_j-1-(1-\eps)\Big(\frac{1+\lambda_j}d-1\Big)\right]\\
	&=\sum_{j=k}^\infty |c_j|^2\left[\alpha_j-1-(1-\eps)\frac{\alpha_j(\alpha_j+d-2)-d+1}d\right]\\
	&=\sum_{j=k}^\infty |c_j|^2\left[\alpha_j-1-\frac{1-\eps}d(\alpha_j-1)(\alpha_j+d-1)\right]\\
	&=-\frac{1-\eps}{d}\sum_{j=k}^\infty |c_j|^2(\alpha_j-1)\Big(\alpha_j-1+d-\frac{d}{1-\eps}\Big),
	\end{align*}
	which proves \eqref{e:pjanic}. We notice that if $\ds\alpha_k-1+d-\frac{d}{1-\eps}\ge 0$, then the same inequality holds for every $j\ge k$ and so $\ds W_0(\tilde{h})-(1-\eps)W_0(z)\le0$, which gives the claim \eqref{e:dybala}.
\end{proof}

\subsection{Measure correction of the competitor}\label{ss:cut} In this section we compute the energy of an  harmonic function after cutting off a ball of radius $\rho/2$ from its support. In particular we will consider the radial cut-off function $\psi^\rho\colon B_1\to [0,1]$ defined by
\begin{equation}\label{e:psirho}
\left\{
\begin{array}{ll}
\ds\psi^\rho(x)=0 & \quad \text{if }|x|\in [0,\rho/2],\\
\\
\ds\Delta \psi^\rho=0& \quad \text{if }|x|\in ]\rho/2,\rho[,\\
\\
\ds\psi^\rho(x)=1 & \quad \text{if }|x|\in [\rho,1].
\end{array}
\right.
\end{equation}
The main result of this subsection is the following. 
\begin{lemma}[Energy of a measure corrected competitor]\label{l:mc_competitor}
	Consider an open set $S\subset\partial B_1$ and a function $\ds \theta\mapsto c(\theta)=\sum_{j=k}^\infty c_j\phi_j(\theta)\in\R^n$, where $\phi_j$ are as in Subsection \ref{ss:tildeh} and $k\ge1$ is fixed. Let $\rho>0$ and $\psi:=\psi^{\rho}$ be as in \eqref{e:psirho}. Let $\ds\tilde{h}(r,\theta)=\sum_{i=k}^{\infty} r^{\alpha_i}\,c_i\,\phi_i(\theta)$ be the harmonic extension of $c$ in the cone  $\bC_S$ defined in \eqref{e:CS}. Then we have
	$$W_0(\psi \tilde h)\le\left(1+\frac{C_0\rho^{2(\alpha_k-1)+d}}{\alpha_k}\right)W_0(\tilde h)+\frac{C_0\rho^{2(\alpha_k-1)+d}}{\alpha_k}\int_{\partial B_1}|c|^2,$$
	where $C_0>0$ is a dimensional constant.
	If moreover $\alpha_k> 1$, then
	\begin{equation}\label{e:gigi}
	W_0(\psi \tilde h)\le\left(1+\frac{C_0\,\alpha_k\,\rho^{2(\alpha_k-1)+d}}{\alpha_k-1}\right)W_0(\tilde h).
	\end{equation}
\end{lemma}

\begin{proof} 
	For any function $f\in H^1(B_1)$ we have 
	\begin{align*}
	\int_{B_1}|\nabla (\psi f)|^2
	&\leq  \int_{B_1}|\psi\nabla f+ f\nabla \psi |^2 
	\leq \, \int_{B_1} \psi^2 |\nabla f|^2 +\, \int_{B_1} \nabla\psi\cdot\nabla(f^2\psi) \\
	&\leq \, \int_{B_1} \psi^2 |\nabla f|^2 +\, \int_{B_{\rho}\setminus B_{\rho/2}} \nabla\psi\cdot\nabla(f^2\psi)\\
	&\leq \, \int_{B_1} |\nabla f|^2 +\, \int_{\partial B_{\rho}} f^2\frac{\partial \psi}{\partial n} \leq \int_{B_1}  |\nabla f|^2 + \frac{C_0}{\rho}\,\int_{\partial B_{\rho}}f^2.
	\end{align*}
	where $C_0$ is a dimensional constant. 
	
	If $f$ is of the form $f(r,\theta)=r^\alpha\phi_j(\theta)$ for some $\alpha>0$ and $j\ge 1$, then we have 
	$$\int_{B_1}  |\nabla f|^2=\int_0^1r^{d-1}\int_{\partial B_1}\Big(\alpha^2r^{2(\alpha-1)}\phi_j^2(\theta)+r^{2\alpha-2}|\nabla_\theta\phi_j|^2\Big)dr\,d\theta=\frac{\alpha^2+\lambda_j}{d+2(\alpha-1)},$$
	so that 
	$$\int_{\partial B_{\rho}}f^2=\rho^{2\alpha+d-1}= \rho^{2\alpha+d-1}\frac{d+2(\alpha-1)}{\alpha^2+\lambda_j}\int_{B_1}  |\nabla f|^2,$$
	which gives 
	\begin{equation}\label{e:pipita}
	\int_{B_1}|\nabla (\psi f)|^2\leq \, \left(1+C_0\rho^{2\alpha+d-2}\frac{d+2(\alpha-1)}{\alpha^2+\lambda_j}\right)\int_{B_1} |\nabla f|^2.
	\end{equation}
	By \eqref{e:pipita}, applied to each component of $\tilde h$, and the orthogonality of $\phi_j$ we have 
	\begin{align*}
	\int_{B_1}|\nabla (\psi \tilde h)|^2&=\sum_{i=k}^\infty \int_{B_1}|\nabla (\psi c_ir^{\alpha_i} \phi_i)|^2\\
	&\le \sum_{i=k}^\infty \left(1+C_0\rho^{2\alpha_i+d-2}\frac{d+2(\alpha_i-1)}{\alpha_i^2+\lambda_i}\right)\int_{B_1}|\nabla (c_ir^{\alpha_i} \phi_i)|^2\\
	&=\sum_{i=k}^\infty \left(1+\frac{C_0\rho^{2\alpha_i+d-2}}{\alpha_i}\right)\int_{B_1}|\nabla (c_ir^{\alpha_i} \phi_i)|^2\\
	&\le\left(1+\frac{C_0\rho^{2\alpha_k+d-2}}{\alpha_k}\right)\sum_{i=k}^\infty \int_{B_1}|\nabla (c_ir^{\alpha_i} \phi_i)|^2=\left(1+\frac{C_0\rho^{2\alpha_k+d-2}}{\alpha_k}\right)\int_{B_1}|\nabla\tilde h|^2.
	\end{align*}
	Now by the definition of $W_0$ and the fact that $\psi\tilde h=\tilde h=c$ on $\partial B_1$ we get 
	\begin{align*}
	W_0(\psi \tilde h)&=\int_{B_1}|\nabla (\psi \tilde h)|^2-\int_{\partial B_1}|c|^2\\
	&\le \left(1+\frac{C_0\rho^{2(\alpha_k-1)+d}}{\alpha_k}\right)\int_{B_1}|\nabla \tilde h|^2-\int_{\partial B_1}|c|^2\\
	&=\left(1+\frac{C_0\rho^{2(\alpha_k-1)+d}}{\alpha_k}\right)W_0(\tilde h)+\frac{C_0\rho^{2(\alpha_k-1)+d}}{\alpha_k}\int_{\partial B_1}|c|^2.
	\end{align*}
	If $\alpha_k>1$, then 
	$$W_0(\tilde h)=\sum_{j=k}^\infty |c_j|^2(\alpha_j-1)\ge (\alpha_k-1)\sum_{j=k}^\infty |c_j|^2= (\alpha_k-1)\int_{\partial B_1}|c|^2,$$
	which implies \eqref{e:gigi}.
\end{proof}

 \section{The epiperimetric inequality}\label{s:epi}
 
This section is dedicated to the proofs of the various epiperimetric inequalities \eqref{e:epi_1}, \eqref{e:epi_13}, \eqref{e:epi_vec_low_density} and \eqref{e:epi_vec}, as sketched in the introduction. First we prove a series of technical lemmas, which corresponds to the different possible lengths of the support $S=\{c>0\}$ of the non-negative trace $c\ge0$, that is $\ds 2\pi -\delta_0\leq |S|$, $\ds |S|\leq \pi-\delta_0$ and $\ds \pi-\delta_0\leq|S|\leq 2\pi-\delta_0$. Most of the results are valid in any dimension $d\ge 2$; only for the last case we will need to assume $d=2$. Finally we will combine the lemmas to prove the various versions of the epiperimetric inequality.

\subsection{Improvement on the very large cones}\label{ss:Sbbig}
In this subsection we consider the case $|S|\ge d\omega_d-\eta_0$, where $d\ge 2$ and $\eta_0>0$ is a sufficiently small dimensional constant.

\begin{lemma}\label{l:very_large_cones}
Let $c\in H^1(\partial B_1)$ and $S=\{c>0\}\subset \partial B_1$. For every $\alpha>0$, there are constants $\eta_0>0$ and $\eps_0>0$, depending only on $\alpha$ and the dimension of the space, such that 
$$\text{if}\quad\int_{\partial B_1}c>\alpha\quad\text{and}\quad |S|\ge d\omega_d-\eta_0\ ,\quad\text{then}\quad W^\square(\tilde h)-\Theta^\square\le (1-\eps_0)\,\left(W^\square(z)-\Theta^\square\right),$$
where $\square=OP,DP$, $\ds \Theta^{OP}=\frac{\omega_d}2$ and $\ds \Theta^{DP}=(\lambda_1+\lambda_2)\frac{\omega_d}2$, and $z$ and $\tilde h$ are respectively the one-homogeneous and the harmonic extensions of $c$ in $B_1$. 
\end{lemma}

\begin{proof}
	Let $\{\phi_j\}_j$ be a complete orthonormal system of eigenfunction on $\partial B_1$ with $\ds\phi_1=(d\omega_d)^{-1/2}$. We decompose the function $c$ as follows
	$$c(\theta)=\sum_{j=1}^{\infty}c_j\phi_j(\theta)=\frac{c_1}{\sqrt{d\omega_d}}+\sum_{j=2}^{\infty}c_j\phi_j(\theta)=:\frac{c_1}{\sqrt{d\omega_d}}+\bar c(\theta).$$
	We use the notation  
	$$\bar z(r,\theta)=r\bar c(\theta)=\sum_{j=2}^{\infty}c_jr\phi_j(\theta)\qquad\text{and}\qquad \bar h(r,\theta)=\sum_{j=2}^{\infty}c_jr^{\alpha_j}\phi_j(\theta).$$
	Thus we have 
	$$z(r,\theta)=\frac{c_1}{\sqrt{d\omega_d}}r+\bar z(r,\theta)\qquad\text{and}\qquad \tilde h(r,\theta)= \frac{c_1}{\sqrt{d\omega_d}}+\bar h(r,\theta).$$
	For the case $\square=OP$, let $0<\eps\le 1/3$ and notice that
	\begin{align*}
	\left(W(\tilde h)-\frac{\omega_d}2\right)-&(1-\eps)\left(W(z)-\frac{\omega_d}2\right)=W_0(\tilde h)-(1-\eps)W_0(z)+\frac{{\omega_d}}{2}-(1-\eps)\left(\frac{\omega_d}2-\frac{\eta_0}d\right)\\
	&=-c_1^2+W_0(\bar h)-(1-\eps)\left(W_0\left(\frac{c_1 r}{\sqrt{d\omega_d}}\right)+W_0(\bar z)\right)+\frac{{\omega_d}}{2}-(1-\eps)\left(\frac{\omega_d}2-\frac{\eta_0}d\right)\\
	&\le-c_1^2-(1-\eps)W_0\left(\frac{c_1 r}{\sqrt{d\omega_d}}\right)+\frac{{\omega_d}}{2}-(1-\eps)\left(\frac{\omega_d}2-\frac{\eta_0}d\right)\\
	&=-c_1^2-(1-\eps)\left(\frac1d-1\right)c_1^2+\frac{\omega_d}{2}-(1-\eps)\left(\frac{\omega_d}2-\frac{\eta_0}d\right)\\
	&=-\left(\frac1d+\frac{\eps(d-1)}d\right)c_1^2+\frac{\eta_0}{d}+\eps\left(\frac{\omega_d}2-\frac{\eta_0}d\right)\le\frac1d\left(-c_1^2+\eta_0+\eps d\omega_d\right).
	\end{align*}  
	On the other hand we have that 
	$$\ds c_1^2=\left(\int_{\partial B_1}c\,\phi_1\right)^2=\frac{1}{d\omega_d}\left(\int_{\partial B_1}c\right)^2\ge \frac{\alpha^2}{d\omega_d},$$
and so, choosing $\eps$ and $\eta_0$ such that $\alpha^2\ge d\omega_d(\eta_0+\eps d\omega_d)$ we get the claim.\\
If $\square =DP$, we have, by similar computations and using $\big|\{\bar h>0\}\big|\leq \omega_d$,
\begin{align*}
\left(W(\tilde h)\right.&\left.-(\lambda_1+\lambda_2)\frac{\omega_d}2\right)-(1-\eps)\left(W(z)-(\lambda_1+\lambda_2)\frac{\omega_d}2\right)\\
	&\leq -\left(\frac1d+\frac{\eps(d-1)}d\right)c_1^2+\lambda_1\left(\frac{\eta_0}{d}+\eps\left(\frac{\omega_d}2-\frac{\eta_0}d\right)\right)+\lambda_2\left(|\{\bar h<0\}|-\frac{\eta_0}{d}-\eps\left(\frac{\omega_d}2-\frac{\eta_0}d\right) \right)\\
	&\le\frac1d\left(-c_1^2+\eta_0+\eps d\omega_d+|\{\bar h<0\}|\right).
	\end{align*}
Now a simple compactness argument on harmonic functions and the maximum principle show that for every $\delta>0$ there exists $\eta_0>0$ small enough such that, if $|S|\geq d\omega_d-\eta_0$, then $|\{\bar h<0\}|\leq \delta$, and so the conclusion follows as before by choosing $\eta_0$ small enough. 
\end{proof}

\subsection{Improvement on the small cones $S_{small}$}\label{ss:small}
In this subsection we consider the situation where $d\ge 2$, $\ds|S|\le \frac{d\omega_d}2-\delta_0$ and $\delta_0>0$ is a dimensional constant. Using the fact that, under these assumptions, the first eigenvalue of $S$ is strictly bigger than $(d-1)$, we can prove the following result directly for vector-valued functions $c$.
 
\begin{lemma}[Small cones]\label{l:small_cones}
Let $n\ge 1$ and $B_1\subset\R^d$ with $d\ge 2$. For every $\delta_0>0$, there are constants $\eps_1,\rho_1>0$, depending on $\delta_0$, $d$ and $n$, such that if the function $c\in H^{1}(\partial B_1; \R^n)\cap C(\partial B_1; \R^n)$, supported on the open set $S=\{|c|>0\}\subset B_1$, is such that $\ds|S|\le \frac{d\omega_d}{2}-\delta_0$, then 
\begin{equation}\label{e:improv_small}
W_0(\psi^{\rho_1} \tilde h)+\lambda\big|\{\psi^{\rho_1} |\tilde h|>0\}\big|\le\left(1-\eps_1\right)\Big(W_0(z)+\lambda\big|\{|z|>0\}\big|\Big) \quad\mbox{for every}\quad\lambda>0\,,
\end{equation}
where $z$ and $\tilde h$ are the one-homogeneous and  the harmonic  extensions of $c$ in $B_1$, defined in \eqref{e:har&rad}, and $\psi^\rho$ is the cut-off function defined in \eqref{e:psirho}.
\end{lemma}

\begin{proof}
We first notice that if $\ds|S|\le \frac{d\omega_d}{2}-\delta_0$, then $\alpha_1:=\alpha_1(S)= 1+\gamma_0$, where $\gamma_0>0$ is a constant depending only on the dimension and $\delta_0$.
By Lemma \ref{l:mc_competitor} and Lemma \ref{l:harm_ext} we have 
\begin{align*}
W_0(\psi^\rho\tilde h)&\le\left(1+\frac{C_0\,\alpha_1\,\rho^{2(\alpha_1-1)+d}}{\alpha_1-1}\right)W_0(\tilde h)\\
&\le\left(1+\frac{C_0\,\rho^{d}}{\alpha_1-1}\right)\left(1-\frac{\alpha_1-1}{d+\alpha_1-1}\right)W_0(z)=\left(1+\frac{C_0\rho^{d}}{\gamma_0}\right)\left(1-\frac{\gamma_0}{d+\gamma_0}\right)W_0(z)\,.
\end{align*}
On the other hand since $\psi=0$ in $B_{\rho/2}$ we get 
$$\big|\{\psi^\rho |\tilde h|>0\}\big|=\left(1-(\rho/2)^d\right)\big|\{|z|>0\}\big|.$$
It is now sufficient to choose $\eps_1$ and $\rho_1$ such that 
\begin{equation}\label{e:inlemmathincone}
1-\frac{\rho_1^d}{2^d}\le 1-\eps_1\qquad\text{and}\qquad
\left(1+\frac{C_0\rho_1^{d}}{\gamma_0}\right)\left(1-\frac{\gamma_0}{d+\gamma_0}\right)\le 1-\eps_1,
\end{equation}
and recall that, since $\alpha_1>1$, then $W_0(z)>0$ (cp. \eqref{e:ringhio}).
\end{proof}

%
%

\subsection{Improvement on the large cones over $S_{big}$}\label{ss:big}
In this subsection we consider arcs $S=S_{big}\subset \partial B_1$ of length $\pi-\delta_0\le |S|\le 2\pi-\delta_0$. The main result is the following

\begin{prop}[Big cones (OP) and (VP)]\label{p:S_big}
Let $B_1\subset\R^2$ and $c\in H^1(\partial B_1;\R^n)$ be a function such that $S:=\{|c|>0\}\subset\partial B_1$ is a connected arc and let $z$ be the one-homogeneous extension of $c$ in $B_1$. For every $\delta_0>0$, there exists a constant $\rho_2>0$, depending only on $\delta_0$, such that the following holds. If $\ds  \pi-\delta_0 \le |S|\le 2\pi-\delta_0$, then for every $\rho\leq \rho_2$ there exists a function $h_\rho\in H^1(B_1;\R^n)$ such that $h_\rho|_{S}=c$, $h_{\rho}=0$ on $\de B_1\setminus S$ and
	\begin{equation}\label{e:improv_big}
	W^{V}(h_\rho)-\frac\pi2\le\left(1-\rho^3\right)\,\left(W^{V}(z)-\frac\pi2\right)\,.
	\end{equation}
\end{prop}

In order to prove this proposition, we distinguish between high and linear frequencies of the boundary datum ad then we sum the respective contributions. In the rest of this subsection we set $W:=W^{V}$.

\subsubsection{The high frequencies on $S_{big}$}
\label{ss:gtheta}
In this subsection we consider the case when the boundary datum $c$ contains only high frequencies. The argument is very close to the one for $S_{small}$, the only difference being that the measure is not involved. The result below holds in any dimension.

\begin{lemma}[High frequencies on $S_{big}$]\label{l:harm_S_big}
Let $\delta_0>0$ and $S\subset\partial B_1$ be an open set such that $|S|\le d\omega_d-\delta_0$ and $\ds c=\sum_{i=2}^\infty c_i\phi_i$, where $\phi_i$ are as in Subsection \ref{ss:tildeh} and $c_i\in \R^n$. Let $z$ and $\tilde h$ be the functions defined in \eqref{e:har&rad}. There are dimensional constants $\eps_3, \rho_3>0$, such that 
$$W_0(\psi^{\rho}\tilde h)\le\left(1-\eps_3\right)W_0(z) \qquad \mbox{for every}\qquad \rho\leq \rho_3,$$
where $\psi^\rho$ is the function from Lemma \ref{l:mc_competitor}.
\end{lemma}
\begin{proof}
We first notice that, as $|S|\leq d\omega_d-\delta_0$, there is a constant $\gamma_0>0$ depending only on the dimension and $\delta_0$ such that 
$$\alpha_2(S)\ge 1+\gamma_0.$$
By Lemma \ref{l:mc_competitor} and Lemma \ref{l:harm_ext}, with $\rho_0:=\rho$, we have 
\begin{align*}
W_0(\psi\tilde h)&\le\left(1+\frac{C_0\rho^{2(\alpha_2-1)+d}}{\alpha_2-1}\right)W_0(\tilde h)\\
&\le\left(1+\frac{C_0\rho^{2(\alpha_2-1)+d}}{\alpha_2-1}\right)\left(1-\frac{\alpha_2-1}{d+\alpha_2-1}\right)W_0(z)\\
&\le\left(1+\frac{C_0\rho^{d}}{\gamma_0}\right)\left(1-\frac{\gamma_0}{d+\gamma_0}\right)W_0(z),
\end{align*}
which, after choosing $\ds \rho\leq \rho_3:= \left(\frac{\gamma_0^2}{2d C_0}\right)^{\sfrac1d}$ and observing that, since $\alpha_2>1$, then $W_0(z)>0$, concludes the proof.
\end{proof}

\subsubsection{The principal frequency on $S_{big}$}\label{ss:phi1}
In this subsection we consider the case when the boundary datum $c$ is of the form $c(\theta)=c_1\phi_1(\theta)$, that is only the first eigenfunction is involved. \emph{ From now on in this subsection we will suppose that the dimension is precisely $d=2$.}
Thus $S$ is an arc of circle and setting $\ds\delta:=|S|-\pi$ we obtain
$$|S|=\pi+\delta,\qquad \lambda_1=\left(\frac{\pi}{\pi+\delta}\right)^2 \qquad\text{and}\qquad\alpha_1=\frac{\pi}{\pi+\delta}.$$
We notice that the case $\delta=0$ is trivial. In fact in this case we have $\alpha_1=1$, $\alpha_2=2$ and choosing $\tilde h$ as in Lemma \ref{l:harm_ext} we have that for $\eps\le 1/3$
$$W_0(\tilde h)\le \left(1-\eps\right)W_0(z)\qquad\text{and}\qquad |\{|\tilde h|>0\}|-\frac\pi2=|\{|z|>0\}|-\frac\pi2=0,$$
which, by the definition of $W$, proves that 
\begin{equation}\label{e:delta=0}
\text{if }\ \delta=0\ \text{ and }\ \eps\le \frac13\ ,\ \text{ then }\ W(\tilde h)\le \left(1-\eps\right)W(z).
\end{equation}
The rest of the section is dedicated to the analogous estimate in the case 
\begin{equation}\label{e:delta}
\delta\in[-\delta_0,0[\, \cup\,]0,\pi-\delta_0]\,,\quad \mbox{where}\quad\delta_0=\delta_0(d)>0.
\end{equation} 

\noindent First we observe that, loosely speaking, $z$ is  a perturbation of size $\delta$ of the flat cone.

\begin{lemma}[Principal frequency on $S_{big}$ I]\label{l:frequenza_1}
Suppose that $\delta\in\R$ is as in \eqref{e:delta}, $S_{big}$ is the arc $\left]0,\pi+\delta\right[$, $s\in H^1(\partial B_1)$ is such that $\{|s|>0\}=\left]0,\pi\right[$ and $z\in H^1(B_1;\R^n)$ is the $1$-homogeneous extension of the function $\bar c\in H^1(\de B_1;\R^n)$, where
\begin{equation}\label{zrtheta}
\ds \bar c(\theta)=\begin{cases} 
 C\, s\left(\ds\frac{\theta\,\pi}{\pi+\delta}\right) &\text{if}\quad\ds \theta\in\left[0,\pi+\delta\right],\\
0&\text{otherwise}\,,
\end{cases}
\end{equation}
with $ C\in \R^n$. Then 
\begin{align}\label{e:firstfreqI}
W(z)-\frac\pi2
&=\frac12 \left(\|\bar c\|^2_{L^2(]0,\pi+\delta[; \R^n)}-\|\bar c\|^2_{L^2(]0,\pi+\delta[;\R^n)}+\delta\right) \notag\\
&=\frac{|C|^2}2\left(\|s'\|_{L^2(]0,\pi[)}^2-\|s\|_{L^2(]0,\pi[)}^2\right)+\frac{\delta}{2}\left(-\frac{|C|^2}{\pi+\delta}\|s'\|_{L^2(]0,\pi[)}^2-\frac{|C|^2}\pi \|s\|_{L^2(]0,\pi[)}^2+1\right).
\end{align}
\end{lemma}

\begin{proof} We shall denote the various $L^2$ norms simply by $\|\cdot\|_2$, the domain beeing the same as the domain of definition of the function inside. Notice that, by the $1$-homogeneity of $z$ we immediately have
\begin{align*}
W_0(z) &=\int_{B_1}|\nabla z|^2\,dx-\|\bar c\|^2_2
=\int_0^1r\,dr\int_{0}^{\pi+\delta}\left(|\partial_r z|^2+\frac{|\partial_\theta z|^2}{r^2}\right)\,d\theta-\|\bar c\|^2_2\\
&=\int_0^1r\,dr\int_{0}^{\pi+\delta}\left(|\bar c|^2+|\bar c'|^2\right)\,d\theta-\|\bar c\|^2_2\\
&=\frac12 \left(\|\bar c'\|^2_2-\|\bar c\|^2_2\right)
\end{align*}
so that, since by definition of $\delta$, $\ds \frac\delta2=|\{|z|>0\}|-\frac\pi2$, the first equality in \eqref{e:firstfreqI} follows.
Next we set $\phi=\ds\frac{\theta\pi}{\pi+\delta}$, we notice that $d\phi\,dr=\ds\frac{\pi}{\pi+\delta}d\theta\,dr$ and we compute
\begin{equation}\label{e:sviluppo_c'}
\|\bar c'\|_2^2=|\tilde c|^2\,\int_{0}^{\pi+\delta}\left(\frac\pi{\pi+\delta}\right)^2|s'|^2\left(\ds\frac{\theta\pi}{\pi+\delta}\right)\,d\theta
=|\tilde c|^2\int_{0}^{\pi}\frac\pi{\pi+\delta}|s'|^2(\phi)\,d\phi
=\frac\pi{\pi+\delta}\,|\tilde c|^2\,\|s'\|_2^2,
\end{equation}
and analogously
\begin{equation}\label{e:sviluppo_c}
\|\bar c\|^2_2=|C|^2\int_{0}^{\pi+\delta}s^2\left(\ds\frac{\theta\pi}{\pi+\delta}\right)\,d\theta=|C|^2\frac{\pi+\delta}\pi\int_{0}^{\pi}s^2(\phi)\,d\phi=\frac{\pi+\delta}{\pi}\,|C|^2\,\|s\|_2^2,
\end{equation}
which immediately gives 
\begin{align*}
W(z)-\frac\pi2 &=\frac{|C|^2}2\left(\frac\pi{\pi+\delta}\|s'\|_2^2-\frac{\pi+\delta}\pi \|s\|_2^2\right)+\frac\delta2\\
&=\frac{|C|^2}2\left(\|s'\|_2^2-\|s\|_2^2\right)+\frac{\delta}{2}\left(-\frac{|C|^2}{\pi+\delta}\|s'\|_2^2-\frac{|C|^2}\pi \|s\|_2^2+1\right).
\end{align*}
\end{proof}

\noindent Next we consider a perturbation $z_\eps$ of the function $z$, by an internal variation of size $\eps$ and we compare the energy $W(z_\eps)$ with the one of $W(z)$.

\begin{lemma}[Principal frequency on $S_{big}$ II]\label{l:frequenza_2}
Suppose that $z\in H^1(B_1;\R^n)$ is the one homogeneous extension of a function $\bar c\in H^1(S;\R^n)$, $S:=]0,\pi+\delta[$, and consider the function $z_\eps\in H^1(B_1;\R^n)$ defined by
$$\ds z_\eps(r,\theta)=\begin{cases}
r\, \bar c\left(\ds\frac{(\pi+\delta)\theta}{\pi+\delta+\eps\xi}\right)&\text{if}\quad \ds \theta\in\left[0,\pi+\delta+\eps \xi(r)\right],\\
0 &\text{otherwise}\,,
\end{cases}$$
where $\xi:[0,1]\to\R^+$ is a smooth function compactly supported on $]0,1[$.
If $\delta+\eps \xi \geq - 2\delta_0$, then 
\begin{align}\label{e:WzepsVSWz}
W(z_\eps)
&\le W(z)+\eps\int_0^1r\xi(r)\,dr\left(1-\frac{\|\bar c\|_{L^2(]0,\pi+\delta[;\R^n)}^2}{\pi+\delta}-\frac{\|\bar c'\|_{L^2(]0,\pi+\delta[;\R^n)}^2}{\pi+\delta}\right)
\notag\\
&\qquad+\eps^2\,\left(\frac{1+(\pi+\delta)^2}{\pi+\delta}\right)\|\bar c'\|_{L^2(]0,\pi+\delta[;\R^n)}^2 \left(\int_0^1 \frac{r\xi(r)^2+r^3|\xi'(r)|^2}{(\pi+\eps \xi)}\,dr\right).
\end{align}
\end{lemma}

\begin{proof}
With an abuse of notation we denote by $z_\eps, \bar c$ the components of the corresponding functions. Moreover we set $\phi(r,\theta):=\ds\frac{(\pi+\delta)\theta}{\pi+\delta+\eps\xi}$ and we compute in polar coordinates
\begin{align*}
|\nabla z_\eps|^2&
=|\partial_r z_\eps|^2+\frac{1}{r^2}|\partial_\theta z_\eps|^2\\
&=\left|\bar c\left(\phi(r,\theta)\right)-\frac{r(\pi+\delta)\,\theta\, \eps \,\xi'(r)}{(\pi+\delta+\eps \xi(r))^2}\bar c'\left(\phi(r,\theta)\right)\right|^2+\left(\ds\frac{\pi+\delta}{\pi+\delta+\eps \xi(r)}\right)^2|\bar c'|^2\left(\phi(r,\theta)\right)\\
&=\underbrace{\bar c^2\left(\phi(r,\theta)\right)+\left(\ds\frac{\pi+\delta}{\pi+\delta+\eps \xi(r)}\right)^2|\bar c'|^2\left(\phi(r,\theta)\right)}_{=:I_1(r,\theta)}
-\underbrace{2\,\eps\frac{r \phi(r,\theta) \xi'(r)}{(\pi+\delta+\eps \xi(r))}\bar c\left(\phi(r,\theta)\right)\bar c'\left(\phi(r,\theta)\right)}_{:=I_2(r,\theta)}\\
&+\underbrace{\eps^2\frac{r^2\phi(r,\theta)^2|\xi'(r)|^2}{(\pi+\delta+\eps\xi(r))^2}|\bar c'|^2\left(\phi(r,\theta)\right)}_{:=I_3(r,\theta)}.
\end{align*}
We notice that 
\begin{align*}
\int_0^1r\int_{0}^{\pi+\delta+\eps\xi(r)} I_1(r,\theta)\,d\theta\,dr 
&=\int_0^1r\,dr\int_{0}^{\pi+\delta+\eps\xi(r)}\left(\bar c^2\left(\phi(r,\theta)\right)+\left(\ds\frac{\pi+\delta}{\pi+\delta+\eps \xi(r)}\right)^2|\bar c'|^2\left(\phi(r,\theta)\right)\right)\,d\theta\\
&=\int_0^1r\,dr\int_0^{\pi+\delta}\left(\frac{\pi+\delta+\eps \xi(r)}{\pi+\delta}\bar c^2\left(\phi\right)+\frac{\pi+\delta}{\pi+\delta+\eps \xi(r)}|\bar c'|^2\left(\phi\right)\right)\,d\phi\\
&=\frac12\left(\|\bar c\|^2_2+\|\bar c'\|^2_2\right)+\eps\frac{\|\bar c\|_2^2}{\pi+\delta}\int_0^1r\xi(r)\,dr -\eps\|\bar c'\|_2^2\int_0^1\frac{r \xi(r)}{\pi+\delta+\eps \xi(r)}\,dr\\
&=\frac12\left(\|\bar c\|^2_2+\|\bar c'\|^2_2\right)+\eps\int_0^1r\xi(r)\,dr\left(\frac{\|\bar c\|_2^2}{\pi+\delta}-\frac{\|\bar c'\|_2^2}{\pi+\delta}\right)\\
&+\eps^2\|\bar c'\|_2^2\int_0^1\frac{r \xi^2(r)}{(\pi+\delta)(\pi+\delta+\eps \xi(r))}\,dr
\end{align*}
where we used $d\phi\,dr =\ds\frac{\pi+\delta}{\pi+\delta+\eps\xi}d\theta\,dr$.
For the second integrand we have
\begin{align*}
\int_0^1r\int_{0}^{\pi+\delta+\eps\,\xi(r)} I_2(r,\theta)\,d\theta\,dr
&=\int_0^1r\,dr\int_{0}^{\pi+\delta+\eps \xi(r)}2\eps \frac{r\xi'(r)}{\pi+\delta+\eps \xi(r)}\phi(r,\theta)\,\bar c\left(\phi(r,\theta)\right)\bar c'\left(\phi(r,\theta)\right)\,d\theta\\
&=\eps\int_0^1r\,dr\int_0^{\pi+\delta}\frac{r \xi'(r)}{(\pi+\delta+\eps\xi(r))}\phi\,(\bar c^2\left(\phi\right))'\frac{\pi+\delta+\eps\xi(r)}{\pi+\delta}\,d\phi\\
&=-\eps\int_0^1r\,dr\int_{S}\frac{r \xi'(r)}{\pi+\delta}\bar c^2\left(\phi\right)\,d\phi=-\frac2{\pi+\delta}\,\eps\,\|\bar c\|_2^2\int_0^1\frac{r^2}2 \xi'(r)\,dr\\
&=2\eps \int_0^1r \xi(r)\,dr\left(\frac{\|\bar c\|_2^2}{\pi+\delta}\right)\,,
\end{align*}
where we used integration by parts in $\phi$ from the second to the third line, together with $\bar c(0)=0=\bar c(\pi+\delta)$, and integration by parts in $r$ in the last equality. Finally for the third integral, we compute
\begin{align*}
\int_0^1r\int_{0}^{\pi+\delta+\eps\xi(r)}&I_3(r,\theta)\,d\theta\,dr
=\eps^2\int_0^1r\,dr\int_{0}^{\pi+\delta+\eps\xi(r)}\frac{r^2|\xi'(r)|^2}{(\pi+\delta+\eps\xi(r))^2}\phi^2|\bar c'|^2\left(\phi(r,\theta)\right)\,d\theta\\
&=\eps^2\int_0^1r\,dr\int_{0}^{\pi+\delta}\frac{r^2|\xi'(r)|^2}{(\pi+\delta+\eps\xi(r))^2}\,\phi^2|\bar c'|^2\left(\phi\right)\frac{\pi+\delta+\eps \xi(r)}{\pi+\delta}\,d\phi\\
&\le\eps^2\left(\pi+\delta\right)^2\|\bar c'\|_2^2\int_0^1\frac{r^3|\xi'(r)|^2}{(\pi+\delta)(\pi+\delta+\eps \xi(r))}\,dr\\
&\le \eps^2\,\|\bar c'\|_2^2\,(\pi+\delta)\int_0^1 \frac{r^3|\xi'(r)|^2}{(\pi+\delta+\eps \xi)}\,dr\,.
\end{align*}
Combining the previous computation, and summing over the components, we conclude for the vectorial functions that
\begin{align*}
W_0(z_\eps)
&\leq \frac12\left(\|\bar c'\|^2_{L^2(]0,\pi+\delta[;\R^n)}-\|\bar c\|^2_{L^2(]0,\pi+\delta[;\R^n)}\right)-\eps\int_0^1r\xi(r)\,dr\left(\frac{\|\bar c\|_2^2}{\pi+\delta}+\frac{\|\bar c'\|_{L^2(]0,\pi+\delta[;\R^n)}^2}{\pi+\delta}\right) \\
&\qquad+\eps^2\,\left(\frac{1+(\pi+\delta)^2}{\pi+\delta}\right)\|\bar c'\|_{L^2(]0,\pi+\delta[;\R^n)}^2 \left(\int_0^1 \frac{r\xi(r)^2+r^3|\xi'(r)|^2}{(\pi+\eps \xi)}\,dr\right)\,.
\end{align*}
Next notice that $z_\eps|_{\de B_1}=\bar c$ and
$$|\{|z_\eps|>0\}|=\int_0^1r(\pi+\delta+\eps \xi(r))\,dr=|\{|z|>0\}|+\eps\int_0^1r\xi(r)\,dr\,,$$
so that, using the first equality of \eqref{e:firstfreqI}, we conclude
\begin{align*}
W(z_\eps)
&\leq W(z)+\eps\int_0^1r\xi(r)\,dr\left(1-\frac{\|\bar c\|_{L^2(]0,\pi+\delta[;\R^n)}^2}{\pi+\delta}-\frac{\|\bar c'\|_{L^2(]0,\pi+\delta[;\R^n)}^2}{\pi+\delta}\right) \\
&\qquad+\eps^2\,\left(\frac{1+(\pi+\delta)^2}{\pi+\delta}\right)\|\bar c'\|_{L^2(]0,\pi+\delta[;\R^n)}^2 \left(\int_0^1 \frac{r\xi(r)^2+r^3|\xi'(r)|^2}{(\pi+\eps \xi)}\,dr\right)
\end{align*}
which is \eqref{e:WzepsVSWz}.
\end{proof}

In the next lemma we combine the estimates of Lemma \ref{l:frequenza_1} and Lemma \ref{l:frequenza_2} to prove the epiperimetric inequality in the case when the trace $\bar c$ is precisely the principal frequency function of the arc $S_{big}$.
 
\begin{lemma}[Principal frequency on $S_{big}$ III]\label{l:frequenza_3}
Suppose that $\delta\in\R$ is as in \eqref{e:delta}, $S_{big}$ is the arc $\left]0,\pi+\delta\right[$, $s:[0,\pi]\to \R$ is such that $\|s\|_2^2=\|s'\|_2^2$ and $\xi:[0,1]\to\R^+$ is a compactly supported function on $[0,1]$. We notice that if $z_\eps$ and $s$ are as in Lemma \ref{l:frequenza_1} and \ref{l:frequenza_2} respectively, than
$$\ds z_\eps(r,\theta)=\begin{cases}
r\,C\,s\left(\ds\frac{\theta\pi}{\pi+\delta+\eps\xi(r)}\right)&\text{if}\quad \theta\in\left[0,\pi+\delta+\eps\xi(r)\right],\\
0 &\text{otherwise}.
\end{cases}$$
If $\eps|\xi|\leq \delta_0$, then
\begin{align*}
W(z_\eps)&\le W(z)+\frac{2\eps}\delta \left(W(z)-\frac\pi2\right)\int_0^1r\xi(r)\,dr\\
&\qquad+\eps\frac{|C|^2\,\|s'\|_2^2}{(\pi+\delta)^2}\left(\delta\int_0^1r\xi(r)\,dr+\eps\frac{\pi\left(1+(\pi+\delta)^2\right)}{(\pi-\delta_0)}\int_0^1\left(r \xi^2+ r^3|\xi'(r)|^2\right)\,dr\right).
\end{align*}
\end{lemma}

\begin{proof} By Lemma \ref{l:frequenza_1}, combined with the condition $\|s\|^2_2=\|s'\|^2_2$, we have that
$$W(z)-\frac\pi2=\frac{\delta}{2}\left(1-\frac{|C|^2}{\pi+\delta}\|s'\|_2^2-\frac{|C|^2}\pi \|s\|_2^2\right)\,.$$
Using this together with \eqref{e:WzepsVSWz}, \eqref{e:sviluppo_c} and \eqref{e:sviluppo_c'}, we obtain
\begin{equation*}
\begin{array}{rcl}
W(z_\eps)- W(z)&
\ds\stackrel{\eqref{e:WzepsVSWz}}{\leq} &\ds\eps\int_0^1r\xi(r)\,dr\left(1-\frac{\|\bar c\|_2^2}{\pi+\delta}-\frac{\|\bar c'\|_2^2}{\pi+\delta}\right)\\
&&\ds\quad+\eps^2\,\left(\frac{1+(\pi+\delta)^2}{\pi+\delta}\right)\|\bar c'\|_2^2 \left(\int_0^1 \frac{r\xi(r)^2+r^3|\xi'(r)|^2}{(\pi+\eps \xi)}\,dr\right)\\
&\ds\stackrel{\eqref{e:sviluppo_c}\&\eqref{e:sviluppo_c'}}{=}&\ds\eps\int_0^1r\xi(r)\,dr\underbrace{\left(1-\frac{|C|^2\,\|s\|_2^2}{\pi}-\frac{|C|^2\,\|s'\|_2^2}{\pi+\delta}\right)}_{=\frac{2}\delta (W(z)-\frac\pi2)}+\eps\,\delta \frac{|C|^2\,\|s'\|^2_2}{(\pi+\delta)^2} \int_0^1r\xi(r)\,dr\\
&&\ds\quad+\eps^2\,\pi\,\left(\frac{1+(\pi+\delta)^2}{(\pi+\delta)^2}\right)|C|^2\,\|s'\|_2^2 \left(\int_0^1 \frac{r\xi(r)^2+r^3|\xi'(r)|^2}{(\pi+\eps \xi)}\,dr\right)\\
&=&\ds\frac{2\eps}\delta \left(W(z)-\frac\pi2\right)\int_0^1r\xi(r)\,dr+\eps\,|C|^2\,\|s'\|_2^2\,\cdot\\
&&\ds\quad\cdot\,\left(\frac{\delta}{(\pi+\delta)^2}\int_0^1r\xi(r)\,dr+\eps\,\pi\, \frac{1+(\pi+\delta)^2}{(\pi+\delta)^2} \int_0^1 \frac{r\xi(r)^2+r^3|\xi'(r)|^2}{(\pi+\eps \xi)}\,dr\right)
\end{array}
\end{equation*}
which, using the bound on $\eps \xi$, gives the claim.
\end{proof}

\subsubsection{Proof of Proposition \ref{p:S_big}} Let $\{\phi_j\}_{j\ge 1}\subset H^1_0(S)$ be the family of eigenfunctions on the arc $S$. Using the same notations of Subsection \ref{ss:tildeh}, we decompose the function $c$ as 
$$c=c_1 \phi_1+g, \quad \mbox{where}\quad g:=\sum_{j=2}^\infty c_j \phi_j\,\mbox{ and }\,c_j\in \R^n \mbox{ for every }j\geq1.$$
Let $z_1,z_g\in H^1(B_1;\R^n)$ be the one-homogeneous extensions in $B_1$ respectively of $c_1\phi_1$ and $g$, and let $\tilde h$ be the harmonic extension of $g$
on the cone generated by $S$, so that 
$$z_1(r,\theta)=r\,c_1\phi_1(\theta)\,,\quad z_g(r,\theta)=r g(\theta)\quad\text{and}\quad \tilde h(r,\theta)=\sum_{j=2}^{+\infty}r^{\alpha_j}\,c_j\phi_j(\theta).$$
Furthermore we choose 
\begin{equation}\label{e:choice_1}
\rho_2\leq \min\left\{\frac{\rho_3}2,\eps_3^{\sfrac13},\tilde C\right\},
\end{equation}
where $\rho_3, \eps_3>0$ are the universal constants of Lemma \ref{l:harm_S_big} and $\tilde C$ will be chosen in \eqref{tildeCchoice}. Let $\rho\leq \rho_2$ and $\psi^{2\rho}$ be the truncation function from Lemma \ref{l:harm_S_big}. Then the truncated function $h^\rho_g:=\psi^{2\rho} \tilde h$ satisfies  
\begin{equation}\label{e:Sbig_1}
W_0(h^\rho_g)\leq (1-\eps_3) W_0(z_g)
\qquad \mbox{and} \qquad
\spt(h^\rho_g)\subset B_1\setminus B_{\rho}\,.
\end{equation}
Moreover, since $\tilde h(\theta)$ and $\phi_1(\theta)$ are orthogonal in $H^1(\de B_1)$ and $\psi^{2\rho}$ is a radial function, $h_g^\rho$ is orthogonal to $\phi_1$ in $H^1(B_1)$.

Up to a change of coordinates we can suppose that $S$ is the arc $\left[0,\pi+\delta\right]$. Next we will apply Lemma \ref{l:frequenza_3} to $\bar c=c_1\phi_1$, $z=z_1$ and $\ds s=\sqrt{\frac{\pi+\delta}{\pi}} \phi_{\pi}$, where $\phi_\pi$ is the first eigenfunction of the semicircle $\ds\phi_\pi(\theta)=\sqrt{\frac2\pi}\sin\theta$, so that indeed $\|s\|_2^2=\|s'\|_2^2$. Let $\xi:[0,1]\to\R^+$ be a smooth positive function with support in $]0,1[$ and such that $\ds\int_0^1r\xi(r)\,dr=\frac12$. Applying Lemma \ref{l:frequenza_3} with the function $\xi_\rho=\rho \xi(r/\rho)$ we obtain
\begin{align*}
\left(W(z_\eps)-\frac\pi2\right)& \le \left(1 +\frac{2\eps}\delta \int_0^1r\xi_\rho(r)\,dr\right) \,\left(W(z_1)-\frac\pi2\right) \\
&\qquad+\eps\frac{\|s'\|_2^2}{(\pi+\delta)^2}\left(\delta\int_0^1r\xi_\rho(r)\,dr+\eps\frac{\pi\left(1+(\pi+\delta)^2\right)}{(\pi-\delta_0)}\int_0^1\left(r \xi_\rho^2+ r^3|\xi_\rho'(r)|^2\right)\,dr\right).
\end{align*}
Choosing $\eps=-\delta$, and recalling that $|\delta|\leq \delta_0$, the previous estimate yields
\begin{align}\label{e:Sbig_2}
W(z_\eps)-\frac\pi2 
&\le \left(1 -2 \int_0^1r\xi_\rho(r)\,dr\right) \,\left(W(z_1)-\frac\pi2\right) \notag\\
&-\delta^2\frac{\|s'\|_2^2}{(\pi+\delta)^2}\left(\int_0^1r\xi_\rho(r)\,dr-\frac{\pi\left(1+(\pi+\delta)^2\right)}{(\pi-\delta_0)}\int_0^1\left(r \xi_\rho^2+ r^3|\xi_\rho'(r)|^2\right)\,dr\right). \notag\\
&= \left(1 -\rho^3\right) \,\left(W(z_1)-\frac\pi2\right)\\
&-\frac{\delta^2\|s'\|_2^2\,\rho^3}{(\pi+\delta_0)^2}\left(\frac12-\rho\frac{\pi\left(1+(\pi+\delta_0)^2\right)}{(\pi-\delta_0)}\int_0^1\left(r \xi^2+ r^3|\xi'(r)|^2\right)\,dr\right)\notag\\
&\le \left(1 -\rho^3\right) \,\left(W(z_1)-\frac\pi2\right),
\end{align}
where in order to have the last inequality we choose 
\begin{equation}\label{tildeCchoice}
\tilde C=\left(\|\xi\|_{L^\infty}+4\frac{\pi\left(1+(\pi+\delta_0)^2\right)}{(\pi-\delta_0)}\int_0^1\left(r \xi^2+ r^3|\xi'(r)|^2\right)\,dr\right)^{-1}\,,
\end{equation}
where $\tilde C$ is a dimensional constant, since $\delta_0$ is universal. Moreover, with this choice of $\tilde C$ we have that $\|\xi_\rho\|_{L^\infty}\leq 1$ and thus the condition 
$$\eps \xi_\rho= \xi_\rho\geq -\delta_0\,,$$
is satisfied and Lemma \ref{l:frequenza_3} can indeed be applied.
\noindent Notice that, since $ \spt(h^{\rho}_g)\subset B_1\setminus B_{\rho}$ and $z_{\eps}(r,\theta)=r\, c_1\, \phi_1(\theta)$ for every $r\geq \rho$ 
we have that $h^\rho_g$ and $z_{\eps}$ are orthogonal in $H^1(B_1;\R^n)$, and therefore summing \eqref{e:Sbig_1} and \eqref{e:Sbig_2} we conclude, with $h^\rho:=z_\eps+h^\rho_g$, that
\begin{align*}
W(h^\rho)-\frac\pi2
&\le W_0(z_\eps)+W_0(h^\rho_g) + |\{|z_\eps|>0\}|-\frac\pi2 = \left(W(z_\eps)-\frac\pi2\right)+W_0(h^\rho_g) \\
&\le \left(1 -\rho^3\right) \,\left(W(z_1)-\frac\pi2\right)+ \left(1- \eps_3\right) W_0(z_g)\\
&=\left(1 -\rho^3\right) \,W_0(z_1)+\left(1- \eps_3\right) W_0(z_g)+ \left(1 -\rho^3\right)\,\left(|\{|z_1|>0\}|-\frac\pi2\right) \\
&\le\left(1 -\rho^3\right) \,W_0(z_1)+\left(1- \rho^3\right) W_0(z_g)+ \left(1 -\rho^3\right)\,\left(|\{|z|>0\}|-\frac\pi2\right) \\
&=\left(1-\rho^3\right) \left(W(z)-\frac\pi2\right),
\end{align*}
where in the first inequality we used that $ \{|h^\rho_g|>0\}\subset \{|z_\eps|>0\}$ so that $|\{|h^\rho|>0\}|\le |\{|z_\eps|>0\}|$ and for the last one we used that $|\{|z_1|>0\}|\leq |\{|z|>0\}|$ and also the fact that, since $\alpha_2>1$, we have that $W_0(z_g)>0$ by \eqref{e:energyz}.
\qed

\subsection{Improvement on the large cones $S_{big}$ for the double phase}\label{ss:big_cones_double}

We can prove an analogous version of Proposition \ref{p:S_big} for the double-phase functional at the points of high density, where both phases are present in the ball $B_1$.

\begin{prop}[Big cones (DP)]\label{p:S_big_DP}
Let  $B_1\subset\R^2$, $\lambda_1,\lambda_2>0$, $\delta_0>0$ and $c\in H^1(\partial B_1)$. Let $S^+:=\{c^+>0\}$ and  $S^-:=\{c^->0\}$ be two disjoint arcs such that $\ds  \pi-\delta_0 \le |S^\pm|\le 2\pi-\delta_0$. There exists a constant $\rho_2>0$, depending only on $\delta_0$, such that for every $0<\rho\leq \rho_2$ there is a function $h_\rho\in H^1(B_1)$ such that $h_\rho|_{S^\pm}=c^\pm$, $h_\rho=0$ on $\de B_1\setminus S$ and
	\begin{equation}\label{e:improv_big_Dp}
	W^{DP}(h_\rho)-(\lambda_1+\lambda_2)\frac\pi2\le\left(1-\rho^3\right)\left(W^{DP}(z)-(\lambda_1+\lambda_2)\frac\pi2\right)\,,
	\end{equation}
	where $z$ is the one-homogeneous extension of $c$ in $B_1$.
\end{prop}

\begin{proof} We are going to implement the procedure from Proposition \ref{p:improvement} to $c^+$ and $c^-$ respectively on $S^+$ and $S^-$. The only additional difficulty is to make sure that the supports of the competitors generated by Lemma \ref{l:frequenza_3} applied to the highest frequencies of $c^+$ and $c^-$ respectively are disjoint.

 Let $\{\phi^\pm_j\}_{j\ge 1}\subset H^1_0(S^\pm)$ be the families of eigenfunctions on the arcs $S^\pm$. Using the same notations of Subsection \ref{ss:tildeh} we set 
$$c_j^+=\int_{\partial B_1}c(\theta)\phi_j^{+}(\theta)\,d\theta\qquad\text{and}\qquad c_j^-=\int_{\partial B_1}c(\theta)\phi_j^{-}(\theta)\,d\theta,$$
and we decompose the functions $c^+$ and $c^-$ as 
$$c^\pm=c_1^\pm \phi_1^\pm+g^\pm, \quad \mbox{where}\quad g^\pm:=\sum_{j=2}^\infty c_j^\pm \phi_j^\pm\,.$$
Let $z_1^\pm,z_g^\pm\in H^1(B_1)$ be the one-homogeneous extensions in $B_1$ respectively of  $c_1^\pm\phi_1^\pm$ and $g^\pm$ and let $\tilde h^\pm$ be the harmonic extension of $g^\pm$
on $\bC_{S^\pm}$, that is 
$$z_1^\pm(r,\theta)=r\,c_1^\pm\phi_1^\pm(\theta)\,,\quad z_g^\pm(r,\theta)=r g^\pm(\theta)\quad\text{and}\quad \tilde h^\pm(r,\theta)=\sum_{j=2}^{+\infty}r^{\alpha_j}\,c_j^\pm\phi_j^\pm(\theta).$$
Furthermore we choose 
\begin{equation}\label{e:choice_12}
\rho_2\leq \min\left\{\frac{\rho_3}2,\eps_3^{\sfrac13},\tilde C\right\},
\end{equation}
where $\rho_3, \eps_3>0$ are the universal constants of Lemma \ref{l:harm_S_big} and $\tilde C$ will be chosen in \eqref{tildeCchoice_double}. Let $\rho\leq \rho_2$ and $\psi^{2\rho}$ be the truncation function from Lemma \ref{l:harm_S_big}. Then the truncated function $h^\pm_g:=\psi^{2\rho} \tilde h^\pm$ satisfies  
\begin{equation}\label{e:Sbig_12}
W_0(h^\pm_g)\leq (1-\eps_3) W_0(z_g^\pm)
\qquad \mbox{and} \qquad
\spt(h^\pm_g)\subset B_1\setminus B_{\rho}\,.
\end{equation}
Moreover, since $\tilde h^\pm(\theta)$ and $\phi_1^\pm(\theta)$ are orthogonal in $H^1(\de B_1)$ and $\psi^{2\rho}$ is a radial function, $h_g^\pm$ is orthogonal to $\phi_1^\pm$ in $H^1(B_1)$. 

\noindent For the linear frequencies we apply Lemma \ref{l:frequenza_3} to $c_1^+ \,\phi_1^+$ and to $c^-_1\,\phi_1^-$ on their respective domains, parametrized as $S^\pm:=[0,\pi+\delta^\pm]$, with the function $s$ being the principal frequency function on the half-sphere $s(\theta)=\ds\phi_\pi(\theta):=\sqrt{\frac2\pi}\sin\theta$ and with $\xi_\rho(r)=\rho \xi(r/\rho)$ the internal variation from Lemma \ref{l:frequenza_3}, to obtain the functions $z_\eps^\pm:B_1\to\R$ satisfying
\begin{align}\label{e:Sbig_2_double}
W(z^\pm_\eps)-\frac\pi2 
&\le \left(1 -\rho^3\right) \,\left(W(z_1^\pm)-\frac\pi2\right),
\end{align}
where we choose $\eps=-\delta$ and 
\begin{equation}\label{tildeCchoice_double}
\tilde C=\left(4\|\xi\|_{L^\infty}+4\frac{\pi\left(1+(\pi+\delta_0)^2\right)}{(\pi-\delta_0)}\int_0^1\left(r \xi^2+ r^3|\xi'(r)|^2\right)\,dr\right)^{-1}\,.
\end{equation}
Furthermore, by this choice, we have $\|\xi_\rho\|_{L^\infty}\leq \frac14$. Now we notice that the set $\partial B_1\setminus (S^+\cup S^-)$ has precisely two connected components and that at least one of them has length greater or equal to $\delta_0/2$. We choose the two internal variations to take place precisely on the boundary of this arc. Thus the supports of the perturbations $z_e^+$ and $z_\eps^-$ are disjoint  
\begin{equation}\label{e:disj_supp}
\spt(z^+_\eps)\cap \spt(z^-_\eps)=\emptyset.
\end{equation}

\noindent Notice that, since $ \spt(h^\pm_g)\subset B_1\setminus B_{\rho}$ and $z^\pm_{\eps}(r,\theta)=r\, c_1^\pm\, \phi_1^\pm(\theta)$, for every $r\geq \rho$, 
we have that $h^\pm_g$ and $z_{\eps}^\pm$ are orthogonal in $H^1(B_1)$, and therefore summing \eqref{e:Sbig_1} and \eqref{e:Sbig_2} we conclude, setting $h_\rho:=h^+-h^-:=(z_\eps^++h^+_g)-(z_\eps^-+h^-_g)$, that
\begin{align*}
W^{DP}(h^\rho)-(\lambda_1+\lambda_2)\frac\pi2 
&\le \left(W_0(z^+_\eps)+W_0(h^+_g) +\lambda_1 |\{z^+_\eps>0\}|-\lambda_1\frac\pi2\right)\\
&\qquad+\left(W_0(z^-_\eps)+W_0(h^-_g) +\lambda_2 |\{z^-_\eps>0\}|-\lambda_2\frac\pi2\right)\\ 
&= \lambda_1\left(W^{OP}(\lambda_1^{-1/2}z_\eps^+)-\frac\pi2\right)+W_0(h^+_g)\\
&\qquad +\lambda_2\left(W^{OP}(\lambda_2^{-1/2}z_\eps^-)-\frac\pi2\right)+W_0(h^-_g)\\
&\le \lambda_1 \left(1 -\rho^3\right)\left(W^{OP}(\lambda_1^{-1/2}z_1^+)-\frac\pi2\right)+ \left(1 -\rho^3\right)W_0(z^+_g)\\
&\qquad +\lambda_2 \left(1 -\rho^3\right)\left(W^{OP}(\lambda_2^{-1/2}z_1^-)-\frac\pi2\right)+ \left(1 -\rho^3\right)W_0(z^-_g)\\
&=\left(1-\rho^3\right) \left(W^{DP}(z)-(\lambda_1+\lambda_2)\frac\pi2\right),
\end{align*}
where in the first inequality we used \eqref{e:disj_supp} to infer that $\spt(h^+)\cap \spt(h^-)=\emptyset$, the choice of $\rho$ and the same observations at the end of the proof of Proposition \ref{p:S_big}.
\end{proof}

\subsection{Proof of Theorem \ref{p:epi_1} for $\mathcal E_{OP}$} 
We are going to denote $W^{OP}$ simply by $W$.
Let $u$ be as in the statement, $c=u|_{\de B_1}$ and let $S:=\spt(c)$. Let $|S|\geq 2\pi-\eta_0$, where $\eta_0$ is the dimensional constant of Lemma \ref{l:very_large_cones}, then \eqref{e:epi_1} follows by the same lemma and the non-degeneracy of $u$ in (ii) Lemma \ref{e:properties_of_min}. 

We now assume that $|S|\le 2\pi-\eta_0$. By the continuity of $u$ (Lemma \ref{e:properties_of_min}) the set $S$ is open and so we can decompose it as the disjoint union of its connected components. Choosing $\ds\delta_0:=\frac{\eta_0}{4}$ we have that there can be at most one connected component of length bigger than $\pi-\delta_0$. Thus we have two possibilities :
$$S=S_{big}\cup \bigcup_{i=1}^\infty S_{small}^i\qquad\text{or}\qquad S=\bigcup_{i=1}^\infty S_{small}^i\ ,$$
where $S_{big}$ and $S^{i}_{small}$, $i\ge1$, are disjoint arcs on $\partial B_1$ such that
\begin{itemize}
\item $S_{big}$ is an arc of length $\pi-\delta_0\le |S_{big}|\le 2\pi-\delta_0$;
\item $S_{small}^i$, for $i\in\N$, are disjoint arcs each one of length $|S^{i}_{small}|\le \pi-\delta_0$.
\end{itemize}
Next we choose
\begin{equation}\label{e:choice_2}
\rho\leq \min\left\{\rho_1,\rho_2,\eps_1^{\sfrac13}\right\},
\end{equation} 
where $\rho_1, \eps_1$ are as in Lemma \ref{l:small_cones} and $\rho_2$ is as in Proposition \ref{p:S_big}, and we distinguish two cases depending on whether $S_{big}$ is empty or not.

Suppose that $S_{big}\neq \emptyset$. Let us denote by $c^i:\partial B_1\to \R$, $i\ge 0$, the traces 
$$c^0:=u|_{S_{big}}\quad\text{and}\quad c^i:=u|_{S^{i}_{small}}\quad\text{for}\quad i\ge 1,$$
by $z^i: B_1\to \R$, $i\ge 0$, the corresponding one-homogeneous extensions 
$$z^0(r,\theta)=r c^0(\theta)\quad\text{and}\quad z^i(r,\theta)=rc^i(\theta)\quad\text{for}\quad i\ge 1,$$
by $h^0$ the function of Proposition \ref{p:S_big} with $\rho$ as in \eqref{e:choice_2} and by $h^i$ the truncated function from Lemma \ref{l:small_cones} with $c=c^i$, $S=S^{i}_{small}$ and truncation function $\psi^{2\rho}$. 
We recall that 
\begin{itemize}
\item for $i\ge 1$, the support of each $h^i$ is contained in the cone over the support of $c^i$,
\item for $i\ge 1$, the choice of the truncation $\psi^{2\rho}$ implies that $h^i$ is zero in $B_\rho$: $\spt(h^i)\subset B_{1}\setminus B_{\rho}$,
\item outside $B_\rho$ the support of $h^0$ is contained in the cone over the support of $c^0$,
\end{itemize}
and so, if $h=\sum_{i=0}^\infty h^i$, we have that
$$
\spt (h) = \spt(h^0)\cup \left(\bigcup_{i=1}^\infty \spt (h^i) \right)  
\quad \mbox{and the union is disjoint}.
$$ 
Summing the energy contributions, we then obtain
\begin{align*}
W(h)
&= \sum_{i=0}^{\infty}W_{0}(h^i)+\left( |\{h^0>0\}|-\frac\pi2   \right)+\sum_{i=1}^{\infty}|\{h^i>0\}|\\
&= W(h^0) + \sum_{i=1}^{\infty}\left(W_0(h^i)+|\{h^i>0\}|\right) \\
& \stackrel{\eqref{e:improv_small}\&\eqref{e:improv_big}}{\leq} \left(1-\rho^3\right) \left(W(z^0) + \sum_{i=1}^{\infty} \left(W_0(z^i)+|\{z^i>0\}|\right) \right) \\
&\leq  \left(1-\rho^3\right) W(z)\,,
\end{align*}
where in the second inequality we used \eqref{e:choice_2} and the positivity of $W_0(z^i)$, $i\geq 1$, to infer that 
$$(1-\eps_1)\left(W_0(z^i)+|\{z^i>0\}|\right) \leq  \left(1-\rho^3\right) \left(W_0(z^i)+|\{z^i>0\}|\right)\,.$$

If $S_{big}= \emptyset$, then with the same notation as above we have, by Lemma \ref{l:small_cones}
\begin{align*}
W(h)-\frac\pi2
&= \sum_{i=1}^{\infty} \left(W_{0}(h^i)+|\{h^i>0\}|\right)-\frac\pi2\\
& \stackrel{\eqref{e:improv_small}}{\leq} (1-\eps_1) \left(\sum_{i=1}^{\infty} \left(W_0(z^i)+|\{z^i>0\}|\right)\right)-\frac\pi2 \leq \left(1-\rho^3\right) W(z)\,.
\end{align*}

\qed

\subsection{Proof of Theorem \ref{p:epi_2}} 
Let $c$ be as in the statement, $c^{\pm}=\max\{\pm c,0\}$, $S^{\pm}:=\{c^{\pm}>0\}\subset\partial B_1$ and $z^{\pm}:B_1\to\R$ be the one-homogeneous extensions of $c^{\pm}$ in $B_1$. 

We start by considering the case when one of the sets $S^+$ and $S^-$ is very large and the other very small. Precisely, we assume that $|S^+|\geq 2\pi-\eta_0$ and $|S^-|\le \eta_0$, where $ \eta_0>0$ the dimensional constant of Lemma \ref{l:very_large_cones}. If $|S^-|=0$, then the conclusion follows by Lemma \ref{l:very_large_cones}. If $0< |S^-|\le \eta_0$, then by Lemma \ref{l:small_cones} there are $\eps_1>0$, $\rho_1>0$ and a function $h^-$ such that 
$$\{h^->0\}\subset\{z^->0\}\,,\quad h^-=0\ \text{ on }\ B_{\rho_1}\,,\quad\text{and}$$
$$W_0(h^-)+\lambda_2|\{h^->0\}\cap B_1|\le (1-\eps_1)\left(W_0(z^-)+\lambda_2|\{z^->0\}\cap B_1|\right).$$
Now we suppose that $\ds (W_0(z^+)+\lambda_1|\{z^+>0\}\cap B_1|-\lambda_1\frac\pi2\ge 0$ (since otherwise the conclusion follows immediately by choosing $z^+-h^-$ as test function). We define the test function $h^+$ as:
$$h^+(x)=\begin{cases}
\tilde h^+_{\rho_1}(x),\ \text{if}\ x\in B_{\rho_1},\\
z^+(x),\ \text{if}\ x\in B_1\setminus B_{\rho_1},
\end{cases}$$
where $\tilde h^+_{\rho_1}$ is the harmonic extension of $z^+$ in the ball $B_{\rho_1}$, that is $\tilde h^+_{\rho_1}(x)=\tilde h^+(\rho_1 x)$. By Lemma \ref{l:very_large_cones} we have that 

\begin{align*}
W_0(h^+)+\lambda_1|\{h^+>0\}\cap B_1|-\lambda_1\frac\pi2&=\rho_1^2\Big(W_0(\tilde h^+)+\lambda_1|\{\tilde h^+>0\}\cap B_1|-\lambda_1\frac\pi2\Big)\\
&\qquad+(1-\rho_1^2)\Big(W_0(z^+)+\lambda_1|\{z^+>0\}\cap B_1|-\lambda_1\frac\pi2\Big)\\
&\le \rho_1^2(1-\eps_0)\left(W_0(z^+)+\lambda_1|\{z^+>0\}\cap B_1|-\lambda_1\frac\pi2\right)\\
&\qquad+(1-\rho_1^2)\Big(W_0(z^+)+\lambda_1|\{z^+>0\}\cap B_1|-\lambda_1\frac\pi2\Big)\\
&= (1-\eps_0\rho_1^2)\left(W_0(z^+)+\lambda_1|\{z^+>0\}\cap B_1|-\lambda_1\frac\pi2\right).
\end{align*}
The claim follows by choosing $\eps=\min\{\eps_1,\eps_0\rho_1^2\}$.

Next, we assume without loss of generality that $|S^+|\le 2\pi-\eta_0$,  $|S^-|\le 2\pi-\eta_0$ and $|S^+|\geq |S^-|$. By the continuity of $c$ the set $S=\{c\neq0\}$ is open and so we can decompose it as a disjoint union of its connected components, on each of which $c$ is either strictly positive or strictly negative. Choosing $\ds\delta_0:=\frac{\eta_0}{4}$ we have that there can be at most one connected component of $S^\pm$ of length bigger than $\pi-\delta_0$. Thus we have three possibilities:
$$ S=S^+_{big}\cup S^-_{big}\cup\bigcup_{i=1}^\infty S_{small}^i\qquad\text{or}\qquad S=S^+_{big}\cup \bigcup_{i=1}^\infty S_{small}^i\qquad\text{or}\qquad S=\bigcup_{i=1}^\infty S_{small}^i\ ,$$
where $S^\pm_{big}$ and $S^{i}_{small}$, $i\ge1$, are disjoint arcs on $\partial B_1$ such that
\begin{itemize}
\item $S^\pm_{big}$ is an arc of length $\pi-\delta_0\le |S^\pm_{big}|\le 2\pi-\delta_0$;
\item $S_{small}^i$, for $i\in\N$, are arcs of length $|S^{i}_{small}|\le \pi-\delta_0$.
\end{itemize}
Next we choose
\begin{equation}\label{e:choice_3}
\rho\leq \min\left\{\rho_1,\rho_2,\eps_1^{\sfrac13}\right\},
\end{equation} 
where $\rho_1, \eps_1$ are as in Lemma \ref{l:small_cones} and $\rho_2$ is as in Proposition \ref{p:S_big}, and we distinguish three cases depending on whether $S^\pm_{big}$ are empty or not.

Suppose that $S^+_{big}\neq \emptyset$ and $S^-_{big}\neq \emptyset$. This is the only new case with respect to the one-phase functional. Let us define
\begin{equation}\label{e:def_c}
c^0:=c|_{S^+_{big}\cup S^-_{big}}\quad\text{and}\quad c^i:=\begin{cases}
+c|_{S^{i}_{small}} & \mbox{if }c|_{S^{i}_{small}}>0\\
-c|_{S^{i}_{small}} & \mbox{if }c|_{S^{i}_{small}}<0
\end{cases}\quad\text{for}\quad i\ge 1,
\end{equation}
by $z^i: B_1\to \R$, $i\ge 0$, the corresponding one-homogeneous extensions 
$$z^0(r,\theta)=r c^0(\theta)\quad\text{and}\quad z^i(r,\theta)=rc^i(\theta)\quad\text{for}\quad i\ge 1,$$
by $h_\rho$ the function of Proposition \ref{p:S_big_DP} with $S^\pm=S^\pm_{big}$ and $\rho$ as in \eqref{e:choice_3} and by $h^i$ the truncated function from Lemma \ref{l:small_cones} with $c=c^i$, $S=S^{i}_{small}$ and truncation function $\psi^{2\rho}$.
We recall that 
\begin{itemize}
\item for $i\ge 1$, the support of each $h^i$ is contained in the cone over the support of $c^i$,
\item for $i\ge 1$, the choice of the truncation $\psi^{2\rho}$ implies that $h^i$ is zero in $B_\rho$: $\spt(h^i)\subset B_{1}\setminus B_{\rho}$,
\item outside $B_\rho$ the support of $h_\rho$ is contained in the cone over the support of $c^\pm$,
\end{itemize}
so that $\spt(h_\rho)\cap \spt(h^i)=\emptyset$ for every $i\geq 1$. 
Let $h:=h_\rho+\sum_{i=1}^\infty(\pm h^i)$, where the sign in front of $h^i$ is the same as the sign in front of $c^i$ in \eqref{e:def_c}, and 
\begin{equation*}
\lambda^i:=\begin{cases}
\lambda_1 & \mbox{if }c|_{S^{i}_{small}}>0\\
\lambda_2 & \mbox{if }c|_{S^{i}_{small}}<0
\end{cases}\quad\text{for}\quad i\ge 1.
\end{equation*}
Then we have 
\begin{align*}
W^{DP}(h)-&(\lambda_1+\lambda_2)\frac\pi2
= \left(W^{DP}(h_\rho)-(\lambda_1+\lambda_2)\frac\pi2\right)+ \sum_{i=1}^{\infty}\left(W_0(h^i)+\lambda^i|\{h^i>0\}|\right) \\
& \stackrel{\eqref{e:improv_small}\&\eqref{e:improv_big}}{\leq} \left(1-\rho^3\right) \left(W^{DP}(z^0)+ \sum_{i=1}^{\infty} \left(W_0(z^i)+\lambda^i|\{z^i>0\}|\right)-(\lambda_1+\lambda_2)\frac\pi2 \right) \\
&\leq  \left(1-\rho^3\right) \left(W^{DP}(z)-(\lambda_1+\lambda_2)\frac\pi2\right)\,,
\end{align*}
where in the second inequality we used \eqref{e:choice_2} and the positivity of each $W_0(z^i)$ to infer that 
$$(1-\eps_1)\left(W_0(z^i)+|\{z^i>0\}|\right) \leq  \left(1-\rho^3\right) \left(W_0(z^i)+|\{z^i>0\}|\right)\,.$$

Next, suppose that $S^+_{big}\neq \emptyset$ and $S^-_{big}=\emptyset$. Then the proof is the same as the one for the one phase, by using Lemma \ref{l:small_cones} for the small arcs $S=S^{i}_{small}$ and Proposition \ref{p:S_big} for $S^+$, and subtracting an additional $\lambda_2\sfrac\pi2$.

Finally, if $S^+_{big}= \emptyset=S^-_{big}$, then the proof follows by Lemma \ref{l:small_cones}, subtracting $(\lambda_1+\lambda_2)\sfrac\pi2$ instead of $\sfrac\pi2$.\\
\qed

\subsection{Proof of Theorem \ref{p:epi_vec}} Let $S:=\{|c|>0\}\subset \de B_1$. We distinguish two cases.
\begin{itemize}
\item[(i)] If there exists a universal constant $\delta_0>0$ such that $|S|\leq 2\pi-\delta_0$, then as in Theorem \ref{p:epi_1} we can write $S$ as a union of disjoint arcs
$$ S=S_{big}\cup \bigcup_{i=1}^\infty S^i_{small}.$$
Now the proof is the same as the one of Theorem \ref{p:epi_1} for the one phase, using Lemma \ref{l:small_cones} and Proposition \ref{p:S_big}.\\
\item[(ii)] If $|S|=2\pi$, then let $h$ be the harmonic extension of $c$ and notice that
$$W^V(h)-\pi=W_0(h)\leq (1-\eps)W_0(z)=(1-\eps)(W^V(z)-\pi)\,.$$ 
Otherwise let $\delta_0>0$ be fixed and decompose $S:=S^0_{big}\cup S^1_{big}\cup \bigcup_{i=2}^\infty S^i_{small}$, where each $S^i_{small}$ is a connected arc of length less than $\pi-\delta_0$, $S^i_{big}$ are connected arcs, and we distinguish the following situations.

If $2\pi-\delta_0\leq |S^0_{big}|< 2\pi$, let $\{\phi_j\}_j$ be a complete orthonormal system of eigenfunction on $S$ and let $c_j\in \R^n$ be the projection of $c$ on $\phi_j$. Moreover set
$$ 
z_1:=r \,c_1\,\phi_1\,,
\qquad z_g:=r\,\sum_{j=2}^\infty c_j\,\phi_j 
\qquad\mbox{and}\qquad
h_g:=\sum_{j=2}^\infty r^{\alpha_j}\,c_j\,\phi_j 
$$
for $c_j\in \R^n$ for every $j\geq 1$.
Then, if $h:=z_1+h_g$, we have by \eqref{e:ringhio} that 
\begin{align*}
W^V(h)-\pi
&=(W^V(z_1)-\pi)+W_0(h_g)
=\frac12\left(|c_1|^2 (\alpha_1^2-1)+|S|-2\pi\right)+W_0(h_g)\\
&\leq \frac12(1-\eps) \left(|c_1|^2 (\alpha_1^2-1)+|S|-2\pi\right)+(1-\eps)\,(W_0(z_g)-\pi)\\
&\leq (1-\eps) \,(W^V(z)-\pi)
\end{align*}
where in the first inequality we used that $\alpha_1^2(S)-1\leq \alpha_1^2(S^0_{big})-1\leq 0$, since $|S|\geq 2\pi-\delta_0$, and $|S|-2\pi\leq 0$.

If $\pi-\delta_0 \leq  |S^1_{big}|\leq |S^0_{big}| \leq  2\pi-\delta_0$, then the proof follows by the same arguments as in the double phase case.

If $  |S^1_{big}| \leq \pi-\delta_0 \leq |S^0_{big}| \leq  2\pi-\delta_0$, then we are in the same situation as in (i), and so the proof follows by the same argument.
\end{itemize} 
\qed

\section{Regularity of the free boundary}

In this section we derive the regularity of the free boundary in a standard way by combining the epiperimetric inequality and the Weiss' monotonicity formula. This is done by first improving the usual monotonicity of $W(u,r)$, giving a rate of convergence to its limit as $r\to 0$. Using this rate we then prove the uniqueness of the blow-up at every point, which, combined with the Lipschitzianity of $u$, will give the smoothness of the free boundary. The main references for this section are \cite{FoSp} and \cite{Weiss2}. 

\subsection{Improvement on Weiss monotonicity formula}
It is well known that for any Lipschitz function $u\in H^1(B_1;\R^n)$ in any dimension the following identity holds 
 \begin{equation}\label{e:Weiss_monotonicity}
\frac{d}{dr}W^\square(u,r)=\frac{d}{r} \big[W^\square(z_r,1)-W^\square(u_r,1)\big]+\frac{1}{r}\int_{\partial B_1}|x\cdot \nabla u_r-u_r|^2\,d\HH^{d-1}\,,
\end{equation}
where $\ds z_r(x):=|x|\,u_r\left(\frac{x}{|x|}\right)$ and $\square=OP,DP, V$ (see for instance \cite{Weiss2} for the one and double-phase in the scalar case and \cite{MaTeVe} for the vectorial case). In dimension two the epiperimetric inequality allows us to improve the Weiss' monotonicity identity. Before stating and proving this improvement, we need a simple lemma that allows us to apply one of the epiperimetric inequalities above uniformly at points with the same density. In particular we recall that $u_{x,r}(y):=r^{-1}\,u(x+ry)$ and we introduce the notation
$$\Gamma^\square_{\theta}(u):=\{x\in \Omega\,:\,\Theta^\square_u(x)=\theta\}\,,$$
where the admissible densities are 
\begin{itemize}
\item if $\square=OP$, then $\theta=\frac\pi2$,
\item if $\square=DP$, then $\theta=\lambda_1\frac\pi2,\,\lambda_2\frac\pi2,\,(\lambda_1+\lambda_2)\frac\pi2$,
\item if $\square=V$, then $\theta=\frac\pi2,\, \pi$.
\end{itemize}

\begin{lemma}\label{l:uniform_radius}
	Let $\Omega\subset\R^2$ be an open set and $u\in H^1(\Omega)$ a minimizer of the functional $\mathcal{E}_V$ in $\Omega$. Then for every compact set $K\Subset \Omega$ and every $\delta_0<\pi$ there exists $r_0>0$ such that  for every $x\in \Gamma^V_{\sfrac\pi2}(u)\cap K$ and every $0<r<r_0$
		\begin{equation}\label{e:vectorial_radius}
		|\{|u_{x,r}|>0\}\cap \de B_1|\leq 2\pi-\delta_0 \,.
		\end{equation}
\end{lemma}	

\begin{proof}
 Assume by contradiction that for some $\delta_0$ there exist a sequence of points $(x_k)_k\subset \Gamma_{\sfrac\pi2}(u)\cap K$ and of radii $r_k\to 0$, such that the sequence $u_k:=u_{x_k,r_k}$ satisfies
\begin{equation}\label{e:cont_1}
	 |\{|u_k|>0\}\cap \de B_1|\geq 2\pi-\delta_0\qquad \forall k\in \N\,. 
\end{equation}
By Lemma \ref{l:properties_of_min_vect}, the Lipschitz constant of the sequence $(u_k)_k$ is uniformly bounded, and so up to a subsequence, we can assume that $u_k\to u_0$ uniformly, and moreover $x_k\to x_0\in \Gamma_{\sfrac\pi2}(u)\cap K$. It is a standard argument to see that each $u_k$ is a minimizer of $\mathcal{E}_V$, so that $u_0$ is 
also a minimizer and
\begin{equation}\label{e:cont_2}
\{|u_k|>0\}\cap B_1\to \{|u_0|>0\}\cap B_1 \qquad \mbox{in the Hausdorff distance,}
\end{equation}	
(see for instance \cite{MaTeVe}). Moreover, by the Weiss monotonicity formula, for every $s>0$, $\rho>0$ and $k$ large enough we have 
\begin{align*}
\frac\pi2\le W^V(u_k,s,0)&=W^V(u,r_ks,x_k)\\
&=\underbrace{W^V(u,r_ks,x_k)-W^V(u,\rho,x_k)}_{\le0}+W^V(u,\rho,x_k)-W^V(u,\rho,x_0)+W^V(u,\rho,x_0)\\
&\le W^V(u,\rho,x_k)-W^V(u,\rho,x_0)+W^V(u,\rho,x_0), 
\end{align*}
that is, passing to the limit as $k\to\infty$, $\pi/2\le W^V(u_0,s,0)\le W^V(u,\rho,x_0)$. Since $\rho$ is arbitrary, we get $W^V(u_0,s,0)=\pi/2$ and using again the Weiss monotonicity formula we obtain that $u_0$ is $1$-homogeneous.
However, the only $1$-homogeneous minimizers with density $\sfrac\pi2$ are the half-plane solutions $u_0(x)=h(x)=\xi\, \max\{0,e\cdot x\}$, see Lemma \ref{l:class_bu_V}, so that 
$$|\{|u_0|>0\}\cap \de B_1|=\pi$$
which together with \eqref{e:cont_1} and \eqref{e:cont_2} gives the desired contradiction.
\end{proof}

\begin{prop}[Decay of the Monotonicity formula]\label{p:improvement}
Suppose that $u$ is a minimizer of the functional $\mathcal E_\square$ in the open set $\Omega\subset\R^2$, where $\square=OP, DP,V$. Then there exists a universal constant $\eps>0$ such that for every compact set $K\Subset \Omega$ there is a constant $C>0$ for which the following inequality holds
$$
W^\square(u,r,x_0)-\Theta_u^\square(x_0)\le C\, r^{\gamma}(W^\square(u,1,x_0)-\Theta_u^\square(x_0))
\qquad 0<r<\dist(K, \de \Omega)\,,\;\forall x_0\in \Gamma^\square_{\theta}(u)\cap K$$ 
where $\theta=\Theta_u^\square(x_0)$ is any of the $2$-dimensional densities $\Theta_u^{OP}(x_0)=\pi/2$, $\Theta_u^{DP}(x_0)=(\lambda_1+\lambda_2)\pi/2$ and $\Theta_u^V(x_0)=\pi/2\ \text{or}\ \pi$, and we have set $\gamma:=\frac{2\eps}{1-\eps}$.
\end{prop}

\begin{proof} Assume w.l.o.g. that $x_0=0$ and let us drop the $\square$. By Lemma \ref{l:uniform_radius} combined with Theorems \ref{p:epi_1}, \ref{p:epi_2} and \ref{p:epi_vec}, for each one of our functionals and every possible density there exists a radius $r_0>0$ such that we can apply the epiperimetric inequality in \eqref{e:Weiss_monotonicity}, to obtain
\begin{align}
\nonumber
\frac{d}{dr}&\left(W(u,r)-\Theta_u(0)\right)
=\frac{2}{r}\big[W^\square(z_r,1)-W(u_r,1)\big]+\frac{1}{r}\int_{\partial B_1}|x\cdot \nabla u_r-u_r|^2\,d\HH^{1}\\ \nonumber
&=\frac{2}{r}\left[W(z_r,1)-\Theta_u(0)-W(u_r,1)+\Theta_u(0)\right]+\frac{1}{r}\int_{\partial B_1}|x\cdot \nabla u_r-u_r|^2\,d\HH^{1}\\ 
&\geq \frac{2}{r}\frac{\eps}{1-\eps}\left(W(u_r,1)-\Theta_u(0)\right)\,+\frac{1}{r}\int_{\partial B_1}|x\cdot \nabla u_r-u_r|^2\,d\HH^{1}\geq \frac{\gamma}{r}\left(W(u_r,1)-\Theta_u(0)\right)\,,\label{tututu}
\end{align}
where we used the minimality of $u_r$ with respect to its boundary datum, the positivity of the last term in \eqref{e:Weiss_monotonicity} and one of the epiperimetric inequalities \eqref{e:epi_1}, \eqref{e:epi_13}, \eqref{e:epi_vec_low_density} or \eqref{e:epi_vec} depending on the density.
Integrating this differential inequality, we conclude that
$$
(W(u,r,x_0)-\Theta_u(x_0))\le C\, r^{\gamma}(W(u,1,x_0)-\Theta_u(x_0))
\qquad \forall 0<r<r_0\,.
$$ 
In order to conclude the proof it is enough to observe that for every $x_0\in \Gamma_\theta(u)\cap K\Subset B_1$
this decay can be derived by the same arguments with a constant $C>0$ which depends only on $W(u,1,x_0)-\Theta_u(x_0)>0$ (by monotonicity) and $\dist(K, \de B_1)$.
\end{proof}

\subsection{Uniqueness of the blow-up limit} Using the decay of $W(u,r, x_0)$ of the previous proposition we can now easily prove that the blow-up limit is unique at every free boundary point.

\begin{prop}\label{p:decay}
	Let $\Omega\subset\R^2$ be an open set and $u\in H^1(\Omega)$ a minimizer of $\mathcal E_\square$, where $\square=OP, DP$ or $V$. Then for every compact set $K\Subset\Omega$, there is a constant $C>0$ such that for every free boundary point $x_0\in \Gamma^\square_{\theta}(u)\cap K$, 
        the following decay holds
	\begin{equation}\label{e:L^2decay}
	\| u_{x_0,t}-u_{x_0,s} \|_{L^2(\partial B_1)} \leq C\, t^{\sfrac{\gamma}{2}}\qquad  \mbox{for all}\quad 0< s<t < \dist (K,\de \Omega)\,,
	\end{equation} 
where $\gamma$ is the exponent from Proposition \ref{p:improvement} and $\theta=\Theta_u^\square(x_0)$ is any of the $2$-dimensional densities $\Theta_u^{OP}(x_0)=\pi/2$, $\Theta_u^{DP}(x_0)=(\lambda_1+\lambda_2)\pi/2$ and $\Theta_u^V(x_0)=\pi/2\ \text{or}\ \pi$.
\end{prop}

\begin{proof}
We are going to treat all the cases at once. Let us assume without loss of generality that $x_0=0$ and let us drop the $\square$. Notice that we can rewrite \eqref{tututu} as
	\begin{equation}\label{e:mon_rem}
    \frac{d}{dr}\left[\frac{W(u,r)-\Theta_u(0)}{r^\gamma}\right]\geq \frac{1}{r^{1+\gamma}}\int_{\partial B_1}|x\cdot \nabla u_r-u_r|^2\,d\HH^{d-1}.
	\end{equation}
Next let $0<s<t<r_0$ and compute
\begin{align}\label{e:imp_2}
\int_{\de B_1}\left| u_t-u_s \right|^2 \,d\,\HH^1
&\leq \int_{\de B_1}\left(\int_s^t\frac{1}{r}\left|  x\cdot \nabla u_r-u_r   \right| \,dr\right)^2\,d\HH^1 \notag\\
&\leq  \int_{\de B_1} \left( \int_s^t r^{-1+\gamma}\,dr\right)\left( \int_s^t r^{-1-\gamma} \left|  x\cdot \nabla u_r-u_r   \right|^2   \,dr\right)\,d\HH^1 \notag\\
&\leq \frac1\gamma(t^\gamma-s^\gamma) \int_s^t r^{-1-\gamma} \int_{\de B_1} \left|  x\cdot \nabla u_r-u_r   \right|^2 \,d\HH^1  \,dr \notag\\
&\stackrel{\eqref{e:mon_rem}}{\leq} \frac{t^\gamma}{\gamma} \int_s^t \frac{d}{dr} \left[\frac{W(u,r)-\Theta_u(0)}{r^\gamma}\right] \,dr\notag\\
&=  \frac{t^\gamma}\gamma\left(\frac{W(u,t)-\Theta_u(0)}{t^\gamma}-\frac{W(u,s)-\Theta_u(0)}{s^\gamma}\right)\le \frac{C t^\gamma}\gamma \,,
\end{align}
where in the last inequality we used the positivity of $W(u,s)-\Theta_u(0)$ and the estimate from Proposition \ref{p:improvement}. 
\end{proof} 

\noindent {\bf Proof of Theorem \ref{t:uniq_blowup}.} In the cases (OP) and (DP) the claim follows immediately from \eqref{e:L^2decay} and the classification of the $2$-dimensional blow-up limits due to Alt-Caffarelli and Alt-Caffarelli-Friedman (see \cite{AlCa,AlCaFr}). For the case (V), again the uniqueness part follows from \eqref{e:L^2decay}, while the classification of the blow-ups from Lemma \ref{l:class_bu_V}. It only remains to prove the last statement of the second bullet.
\noindent Suppose that the free-boundary point $x_0$ is the origin. Moreover assume that the blow-up is such that $e_1\neq e_2$ and that there exists a sequence $(x_k)_k\subset \de\{|u|>0\}$ with $x_k\to 0$. Let $r_k:=|x_k|$ and consider the rescaled functions $u_{r_k}$ and the sequence of points $y_k:=x_k/|x_k|$. By uniqueness of the blow-up $u_{r_k}\to h$ uniformly, and also $y_k\to y\in \mathbb{S}^1$, so that $h(y)=0$. However, since $y=y^1\,e_1+y^2\,e_2$, with at least one of $y_1$, $y_2$ not zero, it follows that $|h|(y)\geq |h_{e_1}|(y)+|h_{e_2}|(y)> 0$, a contradiction.\\
This implies that, in the non isolated points of the free boundary of density $\pi$ the unique tangent function is of the form $h:=\xi\,h_e $, for some $e\in \mathbb{S}^1$. 
\qed

\subsection{Regularity of the one-phase free boundary}\label{ss:one-phase} 
In this subsection we prove that the whole free boundary $\partial \{u>0\}$ is smooth when the scalar function $u$ is a solution of the one-phase functional $\mathcal E_{OP}$. 


\noindent {\bf Proof of Theorem \ref{t:main} (OP).} 
Notice that, by Proposition \ref{p:decay} and Theorem \ref{t:uniq_blowup}, for every $x_0\in \Omega\cap \partial \{u>0\}$, the unique blow-up of the rescaled functions $u_{x_0,r}$ has the form
$$ u_{x_0,0}(x)=:h_{e(x_0)}(x)= \max\{0, x\cdot e(x_0)\},\quad\text{for every}\quad x\in \R^2,$$
where $e(x_0)\in \mathbb{S}^1$. In particular, we have that $\Omega\cap \partial \{u>0\}=\Gamma_{\pi/2}(u)$.

 
 We claim that 
$$\mbox{the function}\quad B_s\cap \de \{|u|>0\} \ni x \mapsto e(x) \in \mathbb{S}^1\quad \mbox{is H\"older continuous.}$$
To see this, let $r:=|x_0-y_0|^{1-\alpha}$, with $\alpha:=\sfrac{\gamma}{(2+\gamma)}$, where $\gamma$ is as in Proposition \ref{p:improvement}. Notice that the Lipschitz continuity of $u$ (see Section \ref{ss:preliminaries}) implies that for every $x\in\partial B_1$ we have
\begin{align*}
\left| u^{x_0}_r(x)-u^{y_0}_r(x) \right|
&\leq r^{-1} \,\int_0^1 \left| (x_0-y_0)\cdot\nabla u\big( s(x_0+rx)+(1-s)(y_0+rx) \big) \right|\,ds\notag\\
&\leq \|\nabla u\|_{L^\infty}\,r^{-1}\,\left|x_0-y_0\right|\leq \|\nabla u\|_{L^\infty}\, \left|x_0-y_0\right|^\alpha\,,
\end{align*}
and so, integrating on $\partial B_1$ and setting $L$ to be the Lipschitz constant of $u$, $L=\|\nabla u\|_{L^\infty}$, we get 
\begin{equation}\label{e:from_lip}
\|u^{x_0}_r-u^{y_0}_r\|_{L^2(\partial B_1)}\le L\, \left|x_0-y_0\right|^\alpha\,.
\end{equation}
On the other hand, it is easy to see that for every pair of vectors $v_1,v_2\in\R^2$ we have
\begin{align}
|v_1-v_2|&=\left(\frac1\pi\int_{\partial B_1}|v_1\cdot x-v_2\cdot x|^2\,dx\right)^{\sfrac12}\notag\\
&\le\left(\frac1\pi\int_{\partial B_1}|(v_1\cdot x)_+-(v_2\cdot x)_+|^2\,dx\right)^{\sfrac12}+\left(\frac1\pi\int_{\partial B_1}|(v_1\cdot x)_--(v_2\cdot x)_-|^2\,dx\right)^{\sfrac12}\notag\\
&=2\left(\frac1\pi\int_{\partial B_1}|(v_1\cdot x)_+-(v_2\cdot x)_+|^2\,dx\right)^{\sfrac12},\label{e:shakaponk}
\end{align}
which gives that 
\begin{equation}\label{e:from_bu}
\left|e(x_0)-e(y_0)\right|\leq \frac2{\sqrt\pi}\|h_{e(x_0)}-h_{e(y_0)}\|_{L^2(\partial B_1)}\,.
\end{equation}
Combining \eqref{e:from_lip}, \eqref{e:from_bu} and \eqref{e:L^2decay} with a triangular inequality, we get
\begin{align}\label{e:normal_hoelder}
\left|e(x_0)-e(y_0)\right|
&\leq 2\, \|h_{e(x_0)}-h_{e(y_0)}\|_{L^2(\partial B_1)}\notag\\
&\leq 2\left(\|u_r^{x_0}-h_{e(x_0)}\|_{L^2(\partial B_1)}+\|u_r^{x_0}-u_r^{y_0}\|_{L^2(\partial B_1)}+\|u_r^{y_0}-h_{e(y_0)}\|_{L^2(\partial B_1)}\right)\notag\\
&\leq 2\,\left(  L\left|x_0-y_0\right|^\alpha+2C r^{\sfrac\gamma2}  \right)=(2L+4C)\,  \left|x_0-y_0\right|^{\alpha}\,.
\end{align}

\noindent Next, for every $x_0\in \Gamma_{\sfrac\pi2}(u)$ and any $\eps>0$, we introduce the cones
$$C^\pm(x_0,\eps):=\{ x\in \R^2\setminus \{0\}\,:\, \pm(x-x_0)\cdot e(x_0)\geq \eps \left|x-x_0\right| \}\,,$$
and we claim that for every $\eps>0$ there exists $\delta>0$ such that for every $x_0\in \Gamma_{\sfrac\pi2}(u)\cap B_{\sfrac s2}$ the following holds:
\begin{equation}\label{e:eisnormal}
\begin{cases}
|u|>0 &\mbox{on }C^+(x_0,\eps)\cap B_\delta(x_0)\\
|u|=0&\mbox{on }C^-(x_0,\eps)\cap B_\delta(x_0)\,,
\end{cases}
\end{equation}
from which the theorem immediately follows as in \cite[Proposition 4.10]{FoSp}. To prove the claim we assume by contradiction that there exists $x_j\in \Gamma_{\sfrac\pi2}\cap B_{\sfrac s2} $ with $x_j\to x_0$ and $y_j\in C^+(x_j,\eps)$ with $|y_j-x_j|\to 0$ such that $|u(y_j)|= 0$. Consider the rescalings $u_j:=u_{x_j,r_j}$, where $r_j:=|x_j-y_j|$, then by the $C^{0,1}$-regularity of $u$
(see (i) Lemma~\ref{l:properties_of_min_vect}) and the fact that we are rescaling geometrically, we deduce that, up to a subsequence, the $u_j$ converges uniformly to $u_0:=u_{x_0,0}=h_{e(x_0)}$. By the H\"older continuity of $e$, we can assume that $\ds\frac{(y_j-x_j)}{r_j}\to z \in C^+(x_0,\eps)\cap \mathbb{S}^1$, which implies that 
$$|u_0|(z)=|h_{e(x_0)}|(z)=|
\max\{0, e(x_0)\cdot z\}|\geq \eps \,|z|=\eps>0
\,.$$
On the other hand, by the uniform convergence of $u_j$ we also have $|u_0|(z)= 0$, which is a contradiction.\qed

\subsection{Regularity of the free boundary of vector-valued minimizers}
This subsection is dedicated to the proof of Theorem \ref{t:main} (V). The argument is precisely the same as in the scalar case, except for the fact that (see Theorem \ref{t:uniq_blowup}) the possible densities at the boundary points are two: $\pi/2$, where the free boundary is smooth and behaves precisely as the free boundary of a scalar one-phase solution, and $\pi$, where the behavior is of double-phase type or cusps may be formed (see Example \ref{ex:cusp_sing}). 

\noindent {\bf Proof of Theorem \ref{t:main} (V).} 
Let $u:\Omega\to\R^n$ be a minimizer of $\mathcal E_{V}$ in the open set $\Omega\subset\R^2$. We recall that the free boundary $\Omega\cap\partial\{|u|>0\}$ can be subdivided into two disjoint sets: 
$$\Omega\cap\partial\{|u|>0\}=\Gamma_{\pi/2}\cup\Gamma_\pi,$$
where 
$$\Gamma_\theta=\Big\{x_0\in \partial\{|u|>0\}\ :\ \Theta_u^V(x_0):=\lim_{r\to0^+}W^V(u,r,x_0)=\theta\Big\}.$$
We first notice that {\it $\Gamma_{\pi/2}$ is an open subset of the free boundary.} 
Since $r\to W^V(u,r,x_0)$ is non-decreasing and $x_0\to W^V(u,r,x_0)$ is continuous, we get that $\Theta_u^V:\Omega\cap\partial\{|u|>0\}\to \R$ is upper semi-continuous and thus $\Gamma_{\pi/2}$ is an open subset of the free boundary. By the uniqueness of the blow-up limits and the Hausdorff convergence of the blow-up sequences (see \cite{MaTeVe}) we get that $\Gamma_{\pi/2}$ coincides precisely with the measure theoretic reduced boundary 
$\partial_{red}\{|u|>0\}$. 

Let $x_0\in\Gamma_{\pi/2}$ and $r_0>0$ be such that $\text{dist}(x_0,\partial\Omega)>r_0$ and $\text{dist}(x_0,\Gamma_{\pi})>r_0$. By Proposition \ref{p:decay} we have that for some constant $C$
\begin{equation}\label{e:L^2decay_1}
	\| u_{x,t}-u_{x,s} \|_{L^2(\partial B_1)} \leq C\, t^{\sfrac{\gamma}{2}}\qquad  \mbox{for all}\quad 0< s<t < \dist (K,\de \Omega)\,,
	\end{equation} 
for every $x\in B_{r_0}(x_0)$. We are now going to prove that {\it the free boundary $\partial\{|u|>0\}$ is $C^{1,\alpha}$ regular in $B_{r_0}(x_0)$. }
Let $x_1,x_2\in B_{r_0}(x_0)$ and let $\xi_1h_{e_1}$ and $\xi_2h_{e_2}$ be the blow-up limits in $x_1$ and $x_2$, where $h_e(x)=\max\{0,x\cdot e\}$, $e_1,e_2\in\partial B_1\subset\R^2$ and $\xi_1,\xi_2\in \partial B_1\subset\R^n$. By \eqref{e:from_bu} we get that 
$$\left|e_1-e_2\right|\leq \frac2{\sqrt\pi}\|h_{e_1}-h_{e_2}\|_{L^2(\partial B_1)}\le \frac2{\sqrt\pi}\|\xi_1h_{e_1}-\xi_2h_{e_2}\|_{L^2(\partial B_1)}.$$
Now reasoning as in \eqref{e:normal_hoelder} we get that 
$$\left|e_1-e_2\right|\le C_0|x_1-x_2|^\alpha,$$
where $C_0$ is a constant depending only on $x_0$ and $r_0$ and $\alpha=\gamma/(2+\gamma)$. Now, by the same argument as in Subsection \ref{ss:one-phase}, $\partial_{red}\{|u|>0\}$ is locally a graph of a $C^{1,\alpha}$ function in $B_{r_0}(x_0)$.\qed

\subsection{Regularity of the free boundary for the double-phase problem}
In this subsection we prove Theorem \ref{t:main} (DP). We are going to show that the normal to the double-phase boundary is $C^{0,\alpha}$, which will imply that the positive and the negative parts of the solution of the double-phase problem are actualy solutions of the classical one-phase free boundary problem in its viscosity formulation (we notice that at this point we will have to apply some result from the classical theory and not Theorem \ref{t:main} (OP) which applies only to variational solutions). 

Let $u\in H^1(\Omega)$ be a local minimizer of the functional $\mathcal E_{DP}$ and suppose that $u$ changes sign in the open set $\Omega\subset\R^2$. We decompose the free boundary $\partial\{u\neq0\}$ as follows: 
$$\partial\{u\neq 0\}=\Gamma_{DP}\cup\Gamma_+\cup\Gamma_-,$$
where $\Gamma_{DP}=\partial\{u>0\}\cap\partial\{u<0\}$, $\Gamma_{+}=\partial\{u>0\}\setminus\partial\{u<0\}$ and $\Gamma_{-}=\partial\{u<0\}\setminus\partial\{u>0\}$.
By the classification of the blow-up limits we have that 
$$\Theta^{DP}_u(x_0):=\lim_{r\to0}W^{DP}(u,x_0,r)=\begin{cases}
(\lambda_1+\lambda_2)\pi/2\quad\text{if}\quad x_0\in\Gamma_{DP},\\
\lambda_1\pi/2\quad\text{if}\quad x_0\in\Gamma_{+},\\
\lambda_2\pi/2\quad\text{if}\quad x_0\in\Gamma_{-}.
\end{cases}$$
Since the function $W^{DP}(u,x,r)$ is continuous in $x$ and monotone in $r$ we have that $\Theta^{DP}_u$ is upper semi-continuos and so $\Gamma_{DP}$ is a closed subset of $\partial\{u\neq 0\}$. As a consequence $\Gamma_+$ and $\Gamma_-$ are open and disjoint. In particular, they are locally the free-boundaries of the solutions $u_+$ and $u_-$ of a one-phase problem. Thus, they are both smooth. We now concentrate our attention at the double-phase boundary $\Gamma_{DP}$. 
\begin{lemma}
Suppose that $u:\Omega\to\R$ is a local minimizer of $\mathcal E_{DP}$ in the open set $\Omega\subset\R^2$ and let $\Gamma_{DP}=\partial\{u>0\}\cap\partial\{u<0\}$. Let $r_0>0$ and $\Omega_{r_0}=\big\{x\in\Omega\ :\ \text{dist}(x,\partial\Omega)>r_0\big\}$. Then, there is a constant $C_0>0$ such that, for every $x_0,y_0\in\Gamma_{DP}\cap \Omega_{r_0}$, we have
\begin{equation}\label{e:altrimenti?}
|\mu_1(x_0)-\mu_1(y_0)|+ |\mu_2(x_0)-\mu_2(y_0)|\le C_0|x_0-y_0|^\alpha\quad\text{and}\quad|e(x_0)-e(y_0)|\le C_0|x_0-y_0|^\alpha,
\end{equation}
where $e(x_0)$, $e(y_0)$ the normal vectors to the free boundary in $x_0$ and $y_0$ and the constants $\mu_1(x_0),\mu_2(x_0),\mu_1(y_0),\mu_2(y_0)$ are determined by the blow-up limits $u_{x_0}$, $u_{y_0}$ of $u$ in $x_0$ and $y_0$, precisely 
$$u_{x_0}(x)=\mu_1(x_0)\max\{0,e(x_0)\cdot x\}+\mu_2(x_0)\min\{0,e(x_0)\cdot x\},$$ 
$$u_{y_0}(x)=\mu_1(y_0)\max\{0,e(y_0)\cdot x\}+\mu_2(y_0)\min\{0,e(y_0)\cdot x\}.$$
In particular, $\Gamma_{DP}$ is locally a closed subset of the graph of a $C^{1,\alpha}$ function.
\end{lemma}
\begin{proof}
We first notice that by Proposition \ref{p:decay} there is a constant $C_0$, depending on $r_0$, such that 
$$\|u_{r,x_0}-u_{x_0}\|_{L^2(\partial B_1)}\le C_0 r^\gamma\quad\text{for every}\quad x_0\in \Gamma_{DP}\cap \Omega_{r_0}\quad\text{and}\quad 0<r<r_0.$$
Now the Lipschitz continuity of $u$ gives that there is a constant (still denoted by $C_0$) such that 
\begin{equation}\label{e:convunif}
\|u_{r,x_0}-u_{x_0}\|_{L^\infty(B_1)}\le C_0 r^\gamma\quad\text{for every}\quad x_0\in \Gamma_{DP}\cap \Omega_{r_0}\quad\text{and}\quad 0<r<r_0.
\end{equation}
Now setting $r:=|x_0-y_0|^{1-\alpha}$ and $\alpha:=\sfrac{\gamma}{(2+\gamma)}$, and reasoning as in \eqref{e:normal_hoelder} we get
\begin{align*}
\|u_{x_0}-u_{y_0}\|_{L^\infty(B_1)}
&\leq \|u_{r,x_0}-u_{x_0}\|_{L^\infty( B_1)}+\|u_{r,x_0}-u_{r,y_0}\|_{L^\infty( B_1)}+\|u_{r,x_0}-u_{x_0}\|_{L^\infty( B_1)}\\
&\leq \left( C_0 r^{\sfrac\gamma2} + \frac{L_0}r{\left|x_0-y_0\right|}+C_0 r^{\sfrac\gamma2}  \right)=(L_0+2C_0)\,  \left|x_0-y_0\right|^{\alpha}\,,
\end{align*}
where $L_0$ is the Lipschitz constant of $u$ in $\Omega_{r_0}$. 
Now using the fact that 
$$\|u_{x_0}^+-u_{y_0}^+\|_{L^\infty(B_1)}\le \|u_{x_0}-u_{y_0}\|_{L^\infty(B_1)},$$
and the inequality \eqref{e:shakaponk} we get that 
$$|\mu_1(x_0)e(x_0)-\mu_1(y_0)e(y_0)|\le C_0 |x_0-y_0|^\alpha.$$
Using the fact that $\lambda_1\le \mu_1(x_0),\mu_1(y_0)\le L_0$, we get that for some constant $C_0$ 
$$|\mu_1(x_0)-\mu_1(y_0)|\le C_0 |x_0-y_0|^\alpha\quad\text{and}\quad |e(x_0)-e(y_0)|\le C_0 |x_0-y_0|^\alpha,$$
which concludes the proof of \eqref{e:altrimenti?}, the argument for $\mu_2$ being analogous. The last claim follows by the same argument as in the proof of Theorem \ref{t:main} (OP). 
\end{proof}

\begin{oss}\label{the_gradient_up_to_the_boundary}
By \eqref{e:convunif} we have that, if $x_0\in \Omega_{r_0}\cap\Gamma_{DP}$, then 
\begin{equation}\label{e:pestorosso}
\big|u(x)-u(x_0)-\mu_1(x_0)(x-x_0)\cdot e(x_0)\big|\le C_0|x-x_0|^{1+\gamma}\quad \text{for every}\quad x\in B_{r_0}(x_0)\cap \{u>0\}.
\end{equation}
In particular, $u$ is differentiable on $\{u>0\}$ up to $x_0$ and $|\nabla u(x_0)|=\mu_1(x_0)$. The analogous result holds on the boundary $\Gamma_+$. In fact, if  $x_0\in \Omega\cap\Gamma_{+}$, then there is some $r_0>0$ such that 
$$\big|u(x)-u(x_0)-\lambda_1(x-x_0)\cdot e(x_0)\big|\le C_0|x-x_0|^{1+\gamma}\quad \text{for every}\quad x\in B_{r_0}(x_0)\cap \{u>0\}.$$
\end{oss}

\begin{lemma}\label{l:visc}
Suppose that $u:\Omega\to\R$ is a local minimizer of $\mathcal E^{DP}$ in the open set $\Omega\subset\R^2$. Then there is a $C^{0,\alpha}$ continuous function $\mu_1:\partial\{u>0\}\cap\Omega\to\R$ such that $\mu_1\ge \lambda_1$ and $u_+$ is a solution of the one-phase problem
$$\Delta u_+=0\quad\text{in}\quad\{u>0\}\,,\qquad |\nabla u_+|=\mu_1\quad\text{on}\quad \partial\{u>0\}\cap\Omega\,,$$
that is, for every $x_0\in \partial \{u>0\}$, there is a unit vector  $e(x_0)\in\partial B_1\subset\R^2$ such that 
$$u_+(x)=\mu_1(x_0)(x-x_0)\cdot e(x_0)+o(|x-x_0|)\quad \text{for every}\quad x\in \{u>0\}.$$
\end{lemma}
\begin{proof}
The existence of a function $\mu_1$ is given by Remark \ref{the_gradient_up_to_the_boundary}. The only point to prove is the $C^{0,\alpha}$ continuity of $\mu_1$. Since $\mu_1$ is H\"older continuous on $\Gamma_{DP}$ and constant on $\Gamma_+$, we just need to show that if $x_0\in\Gamma_{DP}$ is such that there is a sequence  $x_n\in \Gamma_+$ converging to $x_0$ , then $\mu_1(x_0)=\lambda_1$.  Suppose that this is not the case and that $\mu_1(x_0)>\lambda_1$. Let $y_n$ be the projection of $x_n$ on the closed set $\Gamma_{DP}$. 
Setting $r_n=|x_n-y_n|$ and $u_n(x)=\frac1{r_n}u_+(x_n+r_nx)$ we have that $u_n$ is a solution of the free boundary problem 
$$\Delta u_n=0\quad\text{in}\quad\{u_n>0\}\,,\qquad |\nabla u_n|=\lambda_1\quad\text{on}\quad \partial\{u_n>0\}\cap B_1\,.$$
Since $u_n$ are uniformly Lipschitz they converge to a function $u_\infty$ which is also a viscosity solution (see \cite{desilva}) of the same problem. On the other hand, by \eqref{e:pestorosso}, we have that $u_\infty=\mu_1(x_0)\max\{0,x\cdot e(x_0)\}$, which gives that necessarily $\mu_1(x_0)=\lambda_1$.
\end{proof}

\noindent {\bf Proof of Theorem \ref{t:main} (DP).} The proof follows by Lemma \ref{l:visc} and the regularity result for the one-phase problem (see, for example \cite{desilva}).\qed 
 
\subsection{An example of a non-smooth free boundary}
As stated in Example~\ref{ex:cusp_sing}, in this subsection we will show that there exists a local minimizer $u:\R^2\to\R^2$ of the functional $\mathcal E_V$ for which 
\begin{enumerate}[(1)]
\item $\Omega_u=\{|u|>0\}$ is a connencted open set; 
\item there is a point $x_0\in\partial\Omega_u$ of density $\Theta_u^V(x_0)=\pi$. 
\end{enumerate} 
In order to construct a solution with the properties {\it (1)} and {\it (2)} described above, we consider the following situation:
\begin{itemize}
\item Consider the two balls $B':=B_1$ and $B'':=B_1(3,0)$ in $\R^2$.
\item Let $C>0$ be a sufficiently large numerical constant ($C=10$ is one possible choice).
\item Let $\eps\ge 0$ and $u_\eps=(u_\eps^1,u_\eps^2)$ be a solution of the problem
\begin{equation}\label{e:example_problem}
\min\Big\{\mathcal E_V(u)\ :\ u\in H^1(\R^2;\R^2)\,,\ u=((1+\eps)C,C)\ \text{on}\ B'\,,\ u=(C,C)\ \text{on}\ B''\Big\}.
\end{equation}
\item Denote by $\Omega_\eps$ the open set $\{|u_\eps|>0\}$. 
\end{itemize}
We claim the following: 
\begin{enumerate}[(i)]
\item The solutions $u_\eps$ are locally Lipschitz continuous in $\R^2\setminus (B'\cup B'')$ with Lipschitz constants that does not depend on $\eps$. This follows directly by Lemma \ref{l:properties_of_min_vect}.
\item The sets $\Omega_\eps$ are open and connected. 

\noindent{\it Proof:} Suppose that this is not the case. Then $\Omega_\eps$ has two connected components $\Omega_\eps^+$, containing $B'$, and $\Omega_\eps^-$, containing $B''$. Then we have that $u_\eps^1=(1+\eps)u_\eps^2$ on $\Omega_\eps^+$ and $u_\eps^1=u_\eps^2$ on $\Omega_\eps^-$.  Moreover, the function $v_\eps=u_\eps^2$ is a solution of the double-phase problem 
\begin{equation*}\label{e:example_dp_eps_problem}
\min\Big\{\left(1+(1+\eps)^2\right)\int_{R^2}|\nabla v_+|^2+2\int_{R^2}|\nabla v_-|^2+|\{v\neq 0\}|\ :\ v=C\ \text{on}\ B'\,,\ v=-C\ \text{on}\ B''\Big\}.
\end{equation*}

We first notice that the double-phase boundary $\{v_\eps>0\}\cap\{v_\eps<0\}$ is non-empty. Indeed, suppose that this is not the case. Then the sets $\{v_\eps>0\}$ and $\{v_\eps<0\}$ are distant and the functions $v_\eps^+$ and $v_\eps^-$ are local solutions of the one-phase problems
\begin{equation*}\label{e:example_op1_eps_problem}
\min\Big\{\left(1+(1+\eps)^2\right)\int_{R^2}|\nabla v_+|^2+|\{v_+> 0\}|\ :\ v_+=C\ \text{on}\ B'\Big\},
\end{equation*}
\begin{equation*}\label{e:example_op2_eps_problem}
\min\Big\{2\int_{R^2}|\nabla v_-|^2+|\{v_->0\}|\ :\ v_-=C\ \text{on}\ B''\Big\}.
\end{equation*}
In particular, we have that the free boundaries $\partial\{v_\eps>0\}$ and $\partial\{v_\eps<0\}$ are smooth and 
$$|\nabla v_\eps^+|^2=1+(1+\eps)^2\quad\text{on}\quad\partial\{v_\eps>0\}\,,\qquad\text{and}\qquad|\nabla v_\eps^-|^2=2\quad\text{on}\quad\partial\{v_\eps<0\}.$$
Consider the radial test function $\phi_t(x)=t-9\ln |x|$ and let $t_0>0$ be the largest $t$ for which $\phi_t\le v_\eps^+$ on $\R^2$. If $t_0<C$, then there is a free boundary point $x_0\in\partial\{v_\eps>0\}$ such that $\phi_t$ touches $v_\eps^+$ from below in $x_0$ (that is $v_\eps^+-\phi_t$ has a local minimum in $x_0$ and $v_\eps^+(x_0)=\phi_t(x_0)=0$). Then 
$$\frac{9}{|x_0|}=|\nabla \phi_t|(x_0)\le |\nabla v_\eps^+|(x_0)=\sqrt{1+(1+\eps)^2}\le 3.$$
Thus $|x_0|\ge 3$ and so $v_\eps>0$ on the ball $B_3$ which is impossible. Then we have that $t_0=C$ and so, $\{\phi_C>0\}\subset\{v_\eps>0\}$.  On the other hand, the set $\{\phi_C>0\}$ is a ball of radius $R=\exp(C/9)$, which again intersects the ball $B''$ if we choose $C$ large enough. Thus, we have that the double-phase boundary $\partial\{v_\eps>0\}\cap\partial\{v_\eps<0\}$ is non-empty. Moreover, by the same argument, there is a point $x_0\in \partial\{v_\eps>0\}\cap\partial\{v_\eps<0\}$ such that $|\nabla v_\eps^+|(x_0)\ge3$. Now for $\eps$ small enough the optimality condition 
$$|\nabla v_\eps^+|^2-|\nabla v_\eps^-|^2=\frac12-\frac1{1+(1+\eps)^2}\quad\text{on}\quad\partial\{v_\eps>0\}\cap\partial\{v_\eps<0\},$$
implies that also $|\nabla v_\eps^-|(x_0)\ge 2$. By the continuity of the normal derivative, there is a radius $r_0>0$ such that the ball $B_{r_0}(x_0)$ contains only double-phase boundary 
$$B_{r_0}(x_0)\cap \partial \{v_\eps\neq0\}=B_{r_0}(x_0)\cap \partial\{v_\eps>0\}\cap\partial\{v_\eps<0\}.$$
In particular, the set $\{v_\eps=0\}$ has measure zero in $B_1$ and so the functions $u_\eps^1$ and $u_\eps^2$ are both harmonic (and so, smooth) in $B_{r_0}(x_0)$. Now since $u_\eps^1=u_\eps^2$ on $\Omega_\eps^-$ we get that $|\nabla u_\eps^1|(x_0)=|\nabla u_\eps^2|(x_0)$, while since $u_\eps^1=(1+\eps)u_\eps^2$ on $\Omega_\eps^+$ we get that $|\nabla u_\eps^1|(x_0)=(1+\eps)|\nabla u_\eps^2|(x_0)$, which is impossible since by the choice of $x_0$ the gradient is non-zero in this point. Thus, the set $\Omega_\eps$ has to be connected.\qed 
\item Up to a subsequence, $u_\eps$ converges in $H^1(\R^2;\R^2)$ and locally uniformly in $\R^2\setminus (B'\cup B'')$ to a function $u_0$, solution of the problem \eqref{e:example_problem} with $\eps=0$. The uniform convergence follows by the fact that the family of functions is locally uniformly Lipschitz, while the fact that $u_0$ is a minimizer follows by a standard argument, usually applied to blow-up sequences (see for example \cite{MaTeVe}).
\item The function $u_0$ has two equal components $v:=u_0^1=u_0^2$ that are solutions of the double-phase problem 
\begin{equation}\label{e:example_dp_problem}
\min\Big\{\int_{\R^2}|\nabla v|^2+|\{v\neq 0\}|\ :\ v\in H^1(\R^2)\,,\ v=C\ \text{on}\ B'\,,\ v=-C\ \text{on}\ B''\Big\}.
\end{equation}
This claim is just a consequence of the fact that the components of $u_0$ are harmonic functions on the same domain with the same boundary datum.
\item There is a point $x_0\in \partial \{v>0\}\cap \partial\{v<0\}$ such that there are points of the one-phase free boundary $\left(\partial\{v>0\}\setminus\partial\{v<0\}\right)\cup\left(\partial\{v<0\}\setminus\partial\{v>0\}\right)$ arbitrarily close to $x_0$, that is $x_0$ is on the boundary of the one-phase free-boundary. The same argument as the one that we used in the proof of {\it (ii)} shows that the double-phase boundary is non-empty $\partial \{v>0\}\cap \partial\{v<0\}\neq\emptyset$. Moreover, at least one of the boundaries $\partial\{v>0\}\setminus\partial\{v<0\}$ and $\partial\{v<0\}\setminus\partial\{v>0\}$ is non-empty (it is easy to show that the point of $\partial\{v\neq0\}$ at the largest possible distance from zero has to be a one-phase point since the density of the set $\{v\neq0\}$ cannot exceed $1/2$ in that point). 

\item There are $r_0>0$ and $\eps>0$ such that $u_\eps$ has a point of density $\pi$ in $B_{r_0}(x_0)$. Suppose by contradiction that this is not the case. Then, we can apply the epiperimetric inequality at a uniform scale, that is there exist constants $C_0>0$ and $r_0>0$ such that for every $\eps>0$ 
$$\| u_{\eps,x,r}-u_{\eps,x}\|_{L^\infty(B_1)}\le C_0r^\gamma\quad\text{for every}\quad x\in B_{r_0}(x_0),\ 0<r<r_0,$$
where $u_{\eps,x,r}(y)=\frac1ru_\eps(x+ry)$ and $\ds u_{\eps,x}=\lim_{r\to0}u_{\eps,x,r}$. As a consequence, we get 
$$|e_\eps(x)-e_\eps(y)|\le C_0|x-y|^\alpha\quad\text{for every}\quad x,y\in B_{r_0}(x_0),$$
where $e_\eps(x)$ and $e_\eps(y)$ are the exterior normals to $x$ and $y$. On the other hand, let $x_+\in\partial\{v>0\}\setminus\partial\{v<0\}$ be sufficiently close to $x_0$. We notice that for $r>0$ sufficiently small we have $B_r(x_+)\subset B_{r_0}(x_0)$ and 
$$\| u_{0,x_+,r}-u_{0,x_+}\|_{L^\infty( B_1)}\le C_0r^\gamma.$$
By the convergence of $u_\eps\to u_0$ we have
$$\lim_{\eps\to0}\| u_{\eps,x_+,r}-u_{0,x_+,r}\|_{L^\infty(B_1)}=0.$$
Now choosing $\eps>0$ small enough, there are free-boundary points $x_\eps^+\in\partial\Omega_\eps$ arbitrarily close to $x_+$, that is $\ds\lim_{\eps\to0}x_\eps^+=x_+$. Then also $e_0(x_+)=\lim_{\eps\to0} |e_\eps(x_\eps^+)$. In particular, all the free boundaries $\partial\Omega_\eps$ are $C^{1,\alpha}$ graphs in the ball $B_{r_0}(x_0)$ with uniform constants. Thus, also the limit has to be a $C^{1,\alpha}$ graph, which is a contradiction with the fact that the density of the set $\Omega_0=\{|u_0|>0\}$ in $x_0$ is one.
\end{enumerate}

\subsection{Proof of Theorem \ref{c:Q_funct}} We are going to give the proof only for the case $\mathcal{E}^q_{OP}$, the other two cases being analogous. We start by showing that Proposition \ref{p:improvement} holds in this case too. Indeed notice that, if $u$ is a minimizer of $\mathcal{E}^q_{OP}$ in a ball $B_{r}(x_0)\subset \Omega$, then
$$
\int_{B_r(x_0)}|\nabla u|^2+q(x_0)\,|\{u>0\}\cap B_r(x_0)|\leq \int_{B_r(x_0)}|\nabla v|^2+q(x_0)\,|\{v>0\}\cap B_r(x_0)|+C\, r^{2+\alpha}\,,
$$
where we used that $q\in C^{0,\alpha}$ and $C$ depends only on $q$. Rescaling everything by $q(x_0)$, we can assume that $q(x_0)=1$, and so $u$ is an almost minimizer for the functional $\mathcal{E}_{OP}$, that is for every $x_0\in \Omega$ and every $0<r<\dist(x_0, \de \Omega)$ we have
\begin{equation}\label{e:alm_mon}
\int_{B_r(x_0)}|\nabla u|^2+|\{u>0\}\cap B_r(x_0)|\leq \int_{B_r(x_0)}|\nabla v|^2+|\{v>0\}\cap B_r(x_0)|+C\, r^{2+\alpha}\,.
\end{equation}
Next, using \eqref{e:Weiss_monotonicity}, which holds for any function, applied to $u_{r,x_0}$, combined with the almost monotonicity \eqref{e:alm_mon}, we get
$$
\frac{d}{dr}\left(W(u,r)-\frac{\pi}2\right)
\geq \frac{2}{r}\left(\frac\eps{1-\eps}\left(W(u_r,1)-\frac{\pi}2\right)-C\,r^{\alpha}\right)\,.
$$
Let us denote by $\Phi(r):=W(u_r,1)-\frac{\pi}2$, then the previous inequality reads
$$
0\leq \Phi'(r)-\left(\frac{2\eps}{1-\eps}\right)\,r^{-1} \Phi(r)+C\,r^{\alpha-1}$$
so that, multiplying both sides by $r^{-\gamma}$, with $\gamma=\frac{2\eps}{1-\eps}$, we get
$$-C\, r^{\alpha-\gamma-1}\leq \left(\Phi(r)\,r^{-\gamma}\right)'\,.$$
Choosing $\eps$ small enough depending only on $q$, we can assume that $\alpha-\gamma>0$, so that integrating the previous inequality we conclude
$$
\frac{\Phi(s)+C\,s^\alpha}{s^\gamma}\leq\frac{\Phi(t)+C\,t^\alpha}{t^\gamma}\quad\text{for every}\quad 0<s<t\leq r_0\,.
$$
Reasoning as in Proposition \ref{p:improvement}, this implies that there exists a universal constant $\gamma>0$ such that for every compact set $K\Subset \Omega$ there is a constant $C>0$ for which the following inequality holds
$$
W^{OP}(u,r,x_0)-\frac\pi2\le C\, r^{\gamma}\left(W^{OP}(u,1,x_0)-\frac\pi2\right)
\qquad 0<r<\dist(K, \de \Omega)\,,\;\forall x_0\in \de\{u>0\}\cap K\,.
$$
Applying this estimate together with \eqref{e:mon_rem}, and reasoning as in proposition \ref{p:decay}, we immediately conclude that for every compact set $K\Subset\Omega$, there is a constant $C>0$ such that for every free boundary point $x_0\in \de\{u>0\}\cap K$, the following decay holds
\begin{equation}\label{e:L^2decay11}
\| u_{x_0,t}-u_{x_0,s} \|_{L^2(\partial B_1)} \leq C\, t^{\sfrac{\gamma}{2}}\qquad  \mbox{for all}\quad 0< s<t < \dist (K,\de \Omega)\,.
\end{equation} 
Reasoning as in the proofs of Theorems \ref{t:uniq_blowup} and \ref{t:main}, the conclusion easily follow. \\
\qed

\bibliographystyle{plain}
\bibliography{references-Cal}

\end{document}